\RequirePackage{fix-cm}

\documentclass[11pt]{article}

\usepackage{amssymb,amsmath,amsthm}
\usepackage{bm}
\usepackage{enumerate}
\usepackage[affil-it]{authblk}

\newcommand{\rmd}{\mathrm{d}}       
\newcommand{\Lop}{\mathrm{L}}       
\renewcommand{\det}{\mathrm{det}}   

\newcommand{\abs}[1]{\left\lvert#1\right\rvert}     
\newcommand{\absf}[1]{\lvert#1\rvert}     		   
\newcommand{\norm}[1]{\left\lVert#1\right\rVert}    

\newtheorem{theorem}{Theorem}
\newtheorem{lemma}{Lemma}
\newtheorem{proposition}{Proposition}

\theoremstyle{definition}
\newtheorem{definition}{Definition}
\newtheorem{property}{Property}

\begin{document}

\title{Operator-Like Wavelet Bases of $L_2(\mathbb{R}^d)$ \thanks{This research was funded in part by ERC Grant ERC-2010-AdG 267439-FUN-SP and by the Swiss National Science Foundation under Grant 200020-121763.}}

\author{Ildar Khalidov}
\author{Michael Unser}
\author{John Paul Ward}

\affil{ \small Biomedical Imaging Group, Ecole polytechnique f\'ed\'erale de Lausanne (EPFL), CH-1015, Lausanne, Switzerland}

\date{}

\maketitle

\begin{abstract}
The connection between derivative operators and wavelets is well known.  Here we generalize the concept by constructing multiresolution approximations and wavelet basis functions that act like Fourier multiplier operators.  This construction follows from a stochastic model: signals are tempered distributions such that the application of a whitening (differential) operator results in a realization of a sparse white noise.  Using wavelets constructed from these operators, the sparsity of the white noise can be inherited by the wavelet coefficients.  In this paper, we specify such wavelets in full generality and determine their properties in terms of the underlying operator. 

\vspace*{0.25cm}

\vspace*{0.25cm}
\noindent
{\bf MSC classes:} 42C40 \and 42B15 \and 60H15
\end{abstract}

\section{Introduction}\label{sec:intro}

In the past few decades, a variety of wavelets that provide a complete and stable multiscale representation of $L_{2}(\mathbb{R}^d)$ have been developed. The wavelet decomposition is very efficient from a computational point of view, due to the fast filtering algorithm. A fundamental property of traditional wavelet basis functions is that they behave like multiscale derivatives \cite{mallat09,meyer92}.
Our purpose in this paper is to extend this concept by constructing wavelets that behave like a given Fourier multiplier operator $\Lop$, which can be more general than a pure derivative. In our approach, the multiresolution spaces are characterized by generalized B-splines associated with the operator, and we show that, in a certain sense, the wavelet inherits properties of the operator.  Importantly, the operator-like wavelet can be constructed directly from the operator, bypassing the scaling function space.  What makes the approach even more attractive is that, at each scale, the wavelet space is generated by the shifts of a single function. Our work provides a generalization of some known constructions including: cardinal spline wavelets \cite{chui91}, elliptic wavelets \cite{micchelli91}, polyharmonic spline wavelets \cite{vandeville05,vandeville10}, Wirtinger-Laplace operator-like wavelets \cite{vandeville08}, and exponential-spline wavelets \cite{khalidov06}.

In applications, it has been observed that many signals are well represented by a relatively small number of wavelet coefficients.  Interestingly, the model that motivates our wavelet construction explains the origin of this sparsity. The context is that of sparse stochastic processes, which are defined by a stochastic differential equation driven by a (non-Gaussian) white noise.   
Explicitly, the model states that $\Lop s=w$ where the signal $s$ is a tempered distribution, $\Lop$ is a shift-invariant Fourier multiplier operator, and $w$ is a sparse white noise \cite{unser11sm}. The wavelets we construct are designed to act like the operator $\Lop$ so that the wavelet coefficients are determined by a generalized B-spline analysis of $w$.  In particular, we define an interpolating spline $\phi$, corresponding to $\Lop^*\Lop$, from which we derive the wavelets $\psi=\Lop^*\phi$.  Then the wavelet coefficients are formally computed by the $L_2$ inner product
\begin{equation*}
\langle s,\psi \rangle 
= 
\langle s,\Lop^* \phi \rangle  
= 
\langle  \Lop s, \phi \rangle  
= 
\langle  w,\phi \rangle .
\end{equation*}
Sparsity of $w$ combined with localization of the interpolating spline $\phi$ results in sparse wavelet coefficients \cite{unser_part1}.  This model is relevant in medical imaging applications, where good performance has been observed in approximating functional magnetic resonance imaging and positron emission tomography data using operator-like wavelets that are tuned to the hemodynamic or pharmacokinetic response of the system \cite{khalidov11,verhaeghe08}.

Our construction falls under the general setting of pre-wavelets, which are comprehensively covered by de Boor, DeVore, and Ron in \cite{deboor93}.   Two distinguishing properties of our approach are its operator-based nature and the fact that it is non-stationary. Related constructions have been developed for wavelets based on radial basis functions \cite{chui96aw,chui96wa,stoeckler93}.   In fact, \cite{chui96wa} also takes an operator approach; however, the authors were focused on wavelets defined on arbitrarily spaced points.

This paper is organized as follows. In Section \ref{sec:prelim}, we formally define the class of admissible operators and the lattices on which our wavelets are defined. In Section \ref{sec:multires}, we construct the non-stationary multiresolution analysis (MRA) that corresponds to a given operator $\Lop$ and derive approximation rates for functions lying in Sobolev-type spaces. Then, in Section \ref{sec:wave}, we introduce the operator-like wavelets and study their properties; in particular, we derive conditions on $\Lop$ that guarantee that our choice of wavelet yields a stable basis at each scale.  Under an additional constraint on $\Lop$, we use this result to define Riesz bases of $L_2(\mathbb{R}^d)$. In Section \ref{sec:decor}, we prove a decorrelation property for families of related wavelets. In Section \ref{sec:conclusion}, we present connections to prior constructions, and we conclude with some examples of operator-like wavelets.

\section{Preliminaries}\label{sec:prelim}

The primary objects of study in this paper are Fourier multiplier operators and their derived wavelets.  The operators that we consider are shift-invariant operators $\Lop$ that act on $L_2(\mathbb{R}^d)$, the class of square integrable functions $f:\mathbb{R}^d \to \mathbb{C}$.  
The action of such a Fourier multiplier operator is defined by its symbol $\widehat{L}$ in the Fourier domain, with
\begin{equation*}
\Lop f
=
\left( \widehat{L} \widehat{f} \right)^\vee.    
\end{equation*}
The symbol $\widehat{L}$ is assumed to be a measurable function.  The adjoint of $\Lop$ is denoted as $\Lop^*$, and its symbol is the complex conjugate of $\widehat{L}$; i.e., the symbol of $\Lop^*$ is $\widehat{L}^*$. In the previous equation, we used $\widehat{f}$ to denote the Fourier transform of $f$
\begin{equation*}
\widehat{f}(\bm{\omega})
=
\int_{\mathbb{R}^d} f(\bm{x}) e^{-i\bm{x}\cdot\bm{\omega}} {\rm d}\bm{x}. 
\end{equation*}
We use $g^\vee$ to denote the inverse Fourier transform of $g$.
Pointwise values of $\widehat{L}$ are required for some of our analysis, so we restrict the class of symbols by requiring continuity almost everywhere. Additionally, we would like to have a well-defined inverse of the symbol, so $\widehat{L}$ should not be zero on a set of positive measure.  To be precise, we define the class of admissible operators as follows.
\begin{definition}\label{def:admis}
Let $\Lop$ be a Fourier multiplier operator. Then $\Lop$ is admissible if its symbol $\widehat{L}$ is of the form $f/g$, where $f$ and $g$ are continuous functions satisfying:
\begin{enumerate}
\item The set of zeros of $fg$ has Lebesgue measure zero;
\item The  zero sets of $f$ and $g$ are disjoint.
\end{enumerate}
\end{definition}

Notice that each such operator defines a subspace of $L_2$, consisting of functions whose derivatives are also square integrable, and our approximation results focus on these spaces.
\begin{definition}
An admissible operator $\Lop$ defines a Sobolev-type subspace of $L_2(\mathbb{R}^d)$:
\begin{equation*}
W_2^\Lop(\mathbb{R}^d)
:=
\left\{f\in L_2({\mathbb R^d}):\norm{f}_{W_2^{\Lop}}<\infty\right\},
\end{equation*}
where
\begin{equation*}
\norm{f}_{W_2^{\Lop}}
:=
 \left( \int_{\mathbb R^d}\abs{\widehat f(\bm{\omega})}^2 \left(1+\abs{\widehat L(\bm{\omega})}^2\right)\rmd\bm{\omega} \right)^{1/2}.
\end{equation*}
\end{definition}

Having defined the class of admissible operators, we must consider the lattices on which the multiresolution spaces will be defined. It is important to use lattices which are nested, so we consider those defined by an expansive integer matrix.  Specifically, an integer matrix $\mathbf{A}$, whose eigenvalues are all larger than 1 in absolute value, defines a sequence of lattices
\begin{equation*}
\mathbf{A}^j\mathbb{Z}^d
=
\{\mathbf{A}^j\bm{k}:\bm{k}\in \mathbb{Z}^d \}
\end{equation*}
indexed by an integer $j$.    
Using \cite{kalker99} as a reference, we recall some results about lattices generated by a dilation matrix.
First, we know that $\mathbf{A}^{j}\mathbb{Z}^d$ can be decomposed into a finite union of disjoint copies of $\mathbf{A}^{j+1}\mathbb{Z}^d$; there are $\abs{\det(\mathbf{A})}$ vectors $\{\bm{e}_{l}\}_{l=0}^{\abs{\det(\mathbf{A})}-1}$ such that
\begin{equation*}
\bigcup_l \left(\mathbf{A}^j\bm{e}_{l} + \mathbf{A}^{j+1}\mathbb{Z}^d\right) 
= 
\mathbf{A}^{j} \mathbb{Z}^d,
\end{equation*}
and using this notation, our convention will be to set $\bm{e}_{0}=\bm{0}$.

There are also several important properties that arise when using Fourier techniques on more general lattices.  A lattice in the spatial domain corresponds to a dual lattice in the Fourier domain, and the dual lattice of $\mathbf{A}^j\mathbb{Z}^d$ is given by $2\pi(\mathbf{A}^T)^{-j}\mathbb{Z}^d$. Also relevant is the notion of a fundamental domain, which for $\mathbf{A}^j\mathbb{Z}^d$ is a bounded, measurable set $\Omega_j$ satisfying
\begin{equation*}
\sum_{\bm{k} \in \mathbb{Z}^d} \chi_{\Omega_j}(\bm{x}+\mathbf{A}^j\bm{k})
=
1
\end{equation*}
for all $\bm{x}$.
 
In this paper, we restrict our attention to lattices derived from matrices that are constant multiples of orthogonal matrices; i.e., we assume a scaling matrix $\mathbf{A}$ satisfies $\mathbf{A}=a\mathbf{R}$ for some orthogonal matrix $\mathbf{R}$ and constant $a>1$.  The lattices generated by these matrices have some additional nice properties. For example, the lattices generated by $\mathbf{A}$ and $\mathbf{A}^T$ are the same, and the lattices $\mathbf{A}^j\mathbb{Z}^d$ scale uniformly in every direction for $j\in \mathbb{Z}$.  Also, for such matrices, there are only finitely many possible lattices generated by powers of $\mathbf{A}$; i.e., there always exists a positive integer $n$ for which $\mathbf{A}^n=a^n\mathbf{I}$.  
	
In addition to the standard dilation matrices $a\mathbf{I}$ (where $a=2,3,\dots$), there are other matrices satisfying the restriction described above. For example in two dimensions, the quincunx matrix 
\begin{equation*}
\mathbf{A} 
= 
\left(
\begin{matrix}
1&1\\
1&-1
\end{matrix}
\right)
\end{equation*}
is valid, and in three dimensions, one could use
\begin{equation*}
\mathbf{A} 
= 
\left(
\begin{matrix}
2&2&-1\\
2&-1&2\\
-1&2&2
\end{matrix}
\right).
\end{equation*}

\section{Multiresolution Analysis}\label{sec:multires}

The multiresolution framework for wavelet construction was presented by Mallat in the late 1980s \cite{mallat89}.  In the following years, the notion of pre-wavelets was developed, and a more general notion of multiresolution was adopted. We consider this more general setting in order to allow for a wider variety of admissible operators.  
\begin{definition}\label{def:nonstat_mra}
A sequence  $\{V_j\}_{j\in \mathbb{Z}}$ of closed linear subspaces of $L_2(\mathbb{R}^d)$ forms a non-stationary multiresolution analysis if
\begin{enumerate}
\item $V_{j+1} \subseteq V_j$;
\item $\bigcup_{j\in \mathbb{Z}} V_j$ is dense in $L_2(\mathbb{R}^d)$ and $\bigcap_{j\in \mathbb{Z}} V_j$ is at most one-dimensional;
\item $f\in V_j$ if and only if $f(\cdot-\mathbf{A}^j\bm{k})\in V_j$ for all $j\in \mathbb{Z}$ and $\bm{k}\in \mathbb{Z}^d$, where $\mathbf{A}$ is an expansive integer
matrix;
\item For each $j\in \mathbb{Z}$, there is an element $\varphi_j\in V_j$ such that the collection of translates $\{\varphi_j(\cdot-\mathbf{A}^j\bm{k}): \bm{k}\in \mathbb{Z}^d \}$ is a Riesz basis of $V_j$, i.e., there are constants $0<A_j\leq B_j<\infty$ such that 
\begin{equation*}
A_j  \norm{c}_{\ell_2}^2 
\leq  
\norm{\sum_{\bm{k}\in \mathbb{Z}^d} c[\bm{k}] \varphi_j(\cdot-\mathbf{A}^j\bm{k})  }_{L_2(\mathbb{R}^d)}^2      
\leq 
B_j  \norm{c}_{\ell_2}^2.
\end{equation*}
\end{enumerate}
\end{definition}
Let us point out here a few remarks concerning this definition.  First of all, note that we have defined our multiresolution spaces $V_j$ to be `growing' as $j$ approaches $-\infty$. Also, in the second condition we do not require the intersection of the spaces $V_j$ to be $\{0\}$.  Instead, we allow it to be one-dimensional.  This happens, for example, when every space is generated by the dilations of a single function; i.e., there is a $\varphi\in L_{2}(\mathbb{R}^d)$ such that 
\begin{equation*}
V_j 
= 
\left\{\sum_{\bm{k}\in \mathbb{Z}^d} c[\bm{k}]\varphi(\cdot-\mathbf{A}^j \bm{k}): c\in \ell_2(\mathbb{Z}^d)\right\} 
\end{equation*}
for every $j$.

In order to produce non-stationary MRAs, we require additional properties on an admissible operator. Together with a dilation matrix, the operator should admit generalized B-splines (generators of the multiresolution spaces $V_j$) that satisfy decay and stability properties. 
 
As motivation for our definition, let us consider the one-dimensional example where $\Lop$ is defined by 
\begin{equation*}
\Lop f (t) 
= 
\frac{\rmd f}{\rmd t}(t)-\alpha f(t),
\end{equation*}
for some $\alpha>0$. A Green's function for $\Lop$ is $\rho(t)=e^{\alpha t}H(t)$, where $H$ is the Heaviside function.  In order to produce Riesz bases for the scaling matrix $\mathbf{A}=(2)$, we introduce the localization operators $\Lop_{\rmd,j}$ defined by $\Lop_{\rmd,j}f=f-e^{2^j \alpha}f(\cdot-2^j)$.   Then for any $j\in\mathbb{Z}$, the  exponential B-spline $\varphi_j:=\Lop_{\rmd,j}\rho$ is a compactly supported function whose shifts $\{\varphi_j(\cdot-2^jk)\}_{k\in\mathbb{Z}}$ form a Riesz basis.  In the Fourier domain, a formula for $\varphi_j$ is
$\widehat{L}(\omega)^{-1}\widehat{L}_{\rmd,j}(\omega)$, which is 
\begin{equation*}
\widehat{\varphi}_j(\omega) 
= 
\frac{1-e^{2^j(\alpha-i\omega)}}{i\omega-\alpha}.
\end{equation*}
In this form we verify the equivalent Riesz basis condition 
\begin{equation*}
0
<
A_j
\leq  
\sum_{k\in \mathbb{Z}} \abs{\widehat{\varphi}_j(\cdot-2\pi 2^{-j}k)  }^2  
\leq 
B_j
<
\infty.
\end{equation*}
In fact, based on the symbol of $\Lop$, we could have worked entirely in the Fourier domain to determine appropriate periodic functions $\widehat{L}_{\rmd,j}$.  With this example in mind, we make the following definition.
\begin{definition}\label{def:spadm}
We say that an operator $\Lop$ and an integer matrix $\mathbf{D}$ are a spline-admissible pair of order $r>d/2$ if the following conditions are satisfied:
\begin{enumerate}
\item $\Lop$ is an admissible Fourier multiplier operator;
\item $\mathbf{D}=a\mathbf{R}$ with $\mathbf{R}$ an orthogonal matrix and $a>1$;  
\item  There is a constant $C_{\Lop}>0$ such that
\begin{equation*}
C_{\Lop} \left( 1+\abs{\widehat{L}(\bm{\omega})}^2 \right) 
\geq 
\abs{\bm{\omega}}^{2r};
\end{equation*}
\item For every $j\in \mathbb{Z}$, there exists a periodic function $\widehat{L}_{ \mathrm{d},j}$ such that 		$\widehat{\varphi}_j(\bm{\omega}):=\widehat{L}_{\mathrm{d},j}(\bm{\omega}) \widehat{L}(\bm{\omega})^{-1}$ satisfies the Riesz basis condition 
\begin{equation*}
0
<
A_j
\leq 
\sum_{\bm{k}\in\mathbb{Z}^d}\abs{\widehat\varphi_j(\bm{\omega}+2\pi (\mathbf{D}^T)^{-j}\bm{k})}^2
\leq
B_j 
< 
\infty,
\end{equation*}
for some $A_j$ and $B_j$ in $\mathbb{R}$. Here, we require the periodic functions $\widehat{L}_{ \mathrm{d},j}$ to be of the form $\sum_{\bm{k} \in \mathbb{Z}^d} p_j[\bm{k}]e^{i\bm{\omega}\cdot \mathbf{D}^j \bm{k}}$ for some $p\in\ell_1(\mathbb Z^d)$.
\end{enumerate}
\end{definition}
\begin{definition}
Let $\Lop$ and $\mathbf{D}$ be a spline admissible pair.  The functions $\widehat{\varphi}_j$ from Condition 4 of Definition \ref{def:spadm} are in $L_2(\mathbb{R}^d)$, and we refer to the functions
\begin{equation*}
\varphi_j
:=
(\widehat{\varphi}_j)^\vee
\end{equation*}
as generalized B-splines for $\Lop$.
\end{definition} 
\begin{proposition}
Given a spline-admissible pair $\Lop$ and $\mathbf{D}$, the spaces 
\begin{equation*}
V_j
=
\left\{\sum_{\bm{k}\in \mathbb{Z}^d} c[\bm{k}]\varphi_j(\cdot-\mathbf{D}^j \bm{k}): c\in \ell_2(\mathbb{Z}^d)\right\}
\end{equation*}
form a non-stationary MRA.
\end{proposition}
\begin{proof}
The first property of Definition \ref{def:nonstat_mra} is verified using the definition of $\mathbf{D}$ and the Riesz basis conditions on the generalized B-splines $\varphi_j$.  

Density in $L_2(\mathbb{R}^d)$ is a result of the admissibility of $\Lop$, the Riesz basis condition on $\varphi_j$, and the inclusion relation $V_{j+1} \subseteq V_{j}$, cf. \cite[Theorem 4.3]{deboor93}. Also, there is an integer $n$ for which $\mathbf{D}^n=a^n\mathbf{I}$, and the intersection of the spaces $V_{jn}$ is at most one-dimensional by Theorem 4.9 of \cite{deboor93}.

Property 3 follows from the definition of the spaces $V_j$, and lastly, Property 4 of Definition \ref{def:nonstat_mra} is valid due to Property 4 of Definition \ref{def:spadm}. 
\end{proof}

The primary difficulty in proving spline-admissibility is verifying Condition 4, which concerns the existence of generalized B-splines. This problem is closely related to the localization (or `preconditioning') of radial basis functions for the construction of cardinal interpolants \cite{chui92}.  As in that paper, the idea is to construct periodic  
functions $\widehat{L}_{ \mathrm{d},j}$ that cancel the singularities of $\widehat{L}^{-1}$.  In one dimension, we can verify spline-admissibility for any constant-coefficient differential operator. In higher dimensions, spline admissibility holds for the Mat\'ern operators, characterized by $\widehat{L}(\bm{\omega})=(1+\abs{\bm{\omega}}^2)^{\nu/2}$, as they require no localization.  In Section \ref{sec:conclusion}, we provide a less obvious example and show how this Riesz basis property can be verified.  As a final point, note that if one is only interested in analyzing fine-scale spaces, Condition 4 need only be satisfied for $j$ smaller than a fixed integer $j_0$, but in this case, it is necessary to include the space $V_{j_0-1}$ in the wavelet decomposition.

We close this section by determining approximation rates for the multiresolution spaces $\{V_j\}_{j\in \mathbb{Z}}$, in terms of the operator $\Lop$ and the density of the lattices generated by $\mathbf{D}^j$.  In order to state this result, we define the spline interpolants for the operator $\Lop^*\Lop$,
whose symbol is $\absf{\widehat{L}}^2$.  The spline admissibility of this operator is the subject of the next proposition.
\begin{proposition}
If $\Lop$ is spline admissible of order $r>d/2$, then $\Lop^*\Lop$ is spline admissible of order $2r>d$.
\end{proposition}
\begin{proof}
Let $\Lop$ be a spline admissible operator of order $r>d/2$. First notice that $\Lop^*\Lop$ is an admissible Fourier multiplier operator.  Also, we see that $\Lop^*\Lop$ satisfies Condition 3 of Definition \ref{def:spadm} with $r$ replaced by $2r$. Therefore spline admissibility follows if we can exhibit generalized B-splines for $\Lop^*\Lop$ that satisfy the Riesz basis condition, where the integer dilation matrix is the same as for $\Lop$. To that end, let $\varphi_j$ be a generalized B-spline for $\Lop$. Then we claim that $\widehat{\widetilde{\varphi}}_j:=\abs{\widehat{\varphi}_j}^2$ defines the Fourier transform of a generalized B-spline for $\Lop^*\Lop$.  

An upper Riesz bound for $\widetilde{\varphi}_j$ can be found by using the norm inequality between $\ell_1$ and $\ell_2$:
\begin{equation*}
\sum_{\bm{k}\in\mathbb{Z}^d}\abs{\widehat\varphi_j(\bm{\omega}+2\pi (\mathbf{D}^T)^{-j}\bm{k})}^4 
\leq 
\left(\sum_{\bm{k}\in\mathbb{Z}^d}\abs{\widehat\varphi_j(\bm{\omega}+2\pi (\mathbf{D}^T)^{-j}\bm{k})}^2 \right)^2
\leq 
B_j^2.
\end{equation*}

To verify the lower Riesz bound for $\widetilde{\varphi}_j$, we make use of the norm inequality
\begin{equation}\label{eq:L*Lspline1}
\sum_{\abs{\bm{k}}\leq M}\abs{\widehat\varphi_j(\bm{\omega}+2\pi (\mathbf{D}^T)^{-j}\bm{k})}^4 
\geq 
C M^{-d}\left(\sum_{\abs{\bm{k}}\leq M}\abs{\widehat\varphi_j(\bm{\omega}+2\pi (\mathbf{D}^T)^{-j}\bm{k})}^2\right)^2,
\end{equation}
for an appropriately chosen $M>0$. Now, note that the decay condition of spline admissibility implies that for $\abs{\bm{\omega}}$ sufficiently large, there is a constant $C>0$ such that
\begin{equation*}
\abs{\widehat{L}(\bm{\omega})}^{-1} \leq C\abs{\bm{\omega}}^{-r}
\end{equation*}
This decay estimate on $\widehat{L}^{-1}$ combined with the lower Riesz bound for $\varphi_j$ gives
\begin{align*}
A_j 
&\leq
 \sum_{\bm{k}\in\mathbb{Z}^d}\abs{\widehat\varphi_j(\bm{\omega}+2\pi (\mathbf{D}^T)^{-j}\bm{k})}^2\\
&\leq
 \sum_{\abs{\bm{k}}\leq M}\abs{\widehat\varphi_j(\bm{\omega}+2\pi (\mathbf{D}^T)^{-j}\bm{k})}^2 + C\abs{\widehat{L}_{\rm{d},j}(\bm{\omega})}^2 M^{d-2r} \abs{\det(\mathbf{D})}^{2jr/d},
\end{align*}
and hence
\begin{equation}\label{eq:L*Lspline2}
\sum_{\abs{\bm{k}}\leq M}\abs{\widehat\varphi_j(\bm{\omega}+2\pi (\mathbf{D}^T)^{-j}\bm{k})}^2 
\geq 
A_j - C\abs{\widehat{L}_{\rm{d},j}(\bm{\omega})}^2 M^{d-2r} \abs{\det(\mathbf{D})}^{2jr/d}.
\end{equation}

Due to the fact that $2r>d$, we can always choose $M$ large enough to make the right hand side of \eqref{eq:L*Lspline2} positive. Using the estimate \eqref{eq:L*Lspline2} in \eqref{eq:L*Lspline1} establishes a lower Riesz bound for $\widetilde{\varphi}_j$.
\end{proof}

The Riesz basis property of the generalized B-splines for $\Lop^*\Lop$ imply that the $\Lop^*\Lop$-spline interpolants $\phi_j(\bm{x})$, given by
\begin{equation}\label{eq:lagint}
\widehat{\phi}_j(\bm{\omega})
=
\abs{\det(\mathbf{D})}^j\frac{\abs{\widehat{\varphi}_j(\bm{\omega})}^2}{\sum_{\bm{k}\in \mathbb{Z}^d}\abs{\widehat{\varphi}_j(\bm{\omega}+2\pi  (\mathbf{D}^T)^{-j} \bm{k})}^2},
\end{equation}
are well-defined and also generate Riesz bases. Importantly, $\phi_j\in W_2^{\Lop}$  does not depend on the specific choice of the localization operator, as we can see from
\begin{align*}
\widehat{\phi}_j(\bm{\omega}) 
&= 
\abs{\det(\mathbf{D})}^j\frac{\abs{\widehat{L}_{\mathrm{d},j}(\bm{\omega})}^2\abs{\widehat{L}(\bm{\omega})}^{-2}}{\abs{\widehat{L}_{\mathrm{d},j}(\bm{\omega})}^2\sum_{\bm{k} \in \mathbb{Z}^d}\abs{\widehat{L}(\bm{\omega}+2\pi (\mathbf{D}^{T})^{-j}\bm{k})}^{-2}} \\
&= 
\abs{\det(\mathbf{D})}^j \frac{1}{1+\abs{\widehat{L}(\bm{\omega})}^{2}\sum_{\bm{k} \in \mathbb{Z}^d\backslash \{0\}}\abs{\widehat{L}(\bm{\omega}+2\pi (\mathbf{D}^{T})^{-j}\bm{k})}^{-2}}.
\end{align*}

These $\Lop^*\Lop$-spline interpolants play a key role in our wavelet construction, which we describe in the next section; however, for our approximation result, we are more interested in the related functions
\begin{equation}\label{eq:lagmult}
m_j(\bm{\omega})
= 
\frac{\abs{\widehat{L}(\bm{\omega})}^{-2}}{\sum_{\bm{k} \in \mathbb{Z}^d}\abs{\widehat{L}(\bm{\omega}+2\pi (\mathbf{D}^{T})^{-j}\bm{k})}^{-2}},
\end{equation}
which are also needed for the decorrelation result, Theorem \ref{th:dc}.

In order to bound the error of approximation from the spaces $V_j$, we apply the techniques developed in \cite{deboor94}. In that paper, the authors derive a characterization of certain potential spaces in terms of approximation by closed, shift-invariant subspaces of $L_2(\mathbb{R}^d)$.  The same techniques can be applied in our situation, with a few modifications to account for smoothness being determined by different operator norms. 

The error in approximating a function $f\in L_2(\mathbb{R}^d)$ by a closed function space $X$ is denoted by
\begin{equation*}
E(f,X) 
:= 
\min_{s\in X }\norm{f-s}_{L_2(\mathbb{R}^d)},
\end{equation*}
and the approximation rate is given in terms of the density of the lattice in $\mathbb R^d$. The lattice determined by $\mathbf{D}^j$ has density proportional to $\abs{\det(\mathbf{D})}^{j/d}$, so we say that the multiresolution spaces $V_{j}$ provide approximation order $\tilde{r}$ if there is a constant $C>0$ such that    
\begin{equation*}
E(f,V_{j})  
\leq 
C\abs{\det(\mathbf{D})}^{j\tilde{r}/d} \norm{f}_{W_2^\Lop(\mathbb{R}^d)},
\end{equation*}
for every $f\in W_2^\Lop(\mathbb{R}^d)$.

\begin{theorem}
For a spline-admissible pair $\Lop$ and $\mathbf{D}$ of order $r>d/2$, the multiresolution spaces $V_j$ provide approximation order $\tilde{r}\leq r$ if 
\begin{equation*}
\abs{\det (\mathbf{D}^T)}^{-2j\tilde{r}/d}  \frac{  1-m_j(\bm{\omega}) }{1+\abs{\widehat{L}(\bm{\omega})}^2} 
\end{equation*}
is bounded, independently of $j$, in $L_\infty((\mathbf{D}^T)^{-j}\Omega)$, where $\Omega=[-\pi,\pi]^d$.
\end{theorem}
\begin{proof}
This result is a consequence of \cite[Theorem 4.3]{deboor94}. To show this let us introduce the notation $f_j(\cdot)=f(\mathbf{D}^{j}\cdot)$, which implies  that $\widehat{f}_j=\abs{\det(\mathbf{D})}^{-j}\widehat{f} \circ (\mathbf{D}^T)^{-j}$, where $\circ$ denotes composition. The spaces $V_j$ are scaled copies of the integer shift-invariant spaces 
\begin{align*}
V_{j}^j 
&:= 
\{ s(\mathbf{D}^j \cdot): s\in V_j \} \\
&= 
\left\{ \sum_{\bm{k}\in \mathbb{Z}^d} c[\bm{k}]\varphi_j (\mathbf{D}^j (\cdot- \bm{k})): c\in \ell_2(\mathbb Z^d)\right\}.
\end{align*}
We then write the error of approximating a function $f\in W_2^\Lop(\mathbb{R}^d)$ from $V_j$ in terms of approximation by $\widehat{V}_j^j$ as
\begin{align*}
E(f,V_j) 
&= 
\abs{\det (\mathbf{D})}^{j/2}E(f_j,V_j^j) \\
&= 
(2\pi)^{-d/2}\abs{\det (\mathbf{D})}^{j/2}E(\widehat{f}_j,\widehat{V}_j^j),
\end{align*}
where $\widehat{V}_j^j$ is composed of the Fourier transforms of functions in $V_j^j$.  Separating this last term, we have
\begin{equation}\label{eq:approx1}
E(f,V_j) 
\leq  
(2\pi)^{-d/2}\abs{\det(\mathbf{D})}^{j/2}\left( E(\widehat{f}_j\chi_\Omega,\widehat{V}_j^j) +\norm{(1-\chi_\Omega)\widehat{f}_j}_2 \right),
\end{equation}
where $\chi_\Omega$ is the characteristic function of the set $\Omega$. We are now left with bounding both terms on the right-hand side of \eqref{eq:approx1}. First, we have
\begin{align*}
\norm{(1-\chi_\Omega)\widehat{f}_j}_2^2 
&= 
\int_{\mathbb R^d\backslash \Omega} \abs{\widehat{f}_j(\bm{\omega})}^2 \rmd\bm{\omega} \\
&= 
\abs{\det(\mathbf{D})}^{-2j} \int_{\mathbb R^d\backslash \Omega} \abs{\widehat{f}((\mathbf{D}^T)^{-j}\bm{\omega})}^2 \frac{1+\abs{\widehat{L}((\mathbf{D}^T)^{-j}\bm{\omega})}^2}{1+\abs{\widehat{L}((\mathbf{D}^T)^{-j}\bm{\omega})}^2} \rmd\bm{\omega}, 
\end{align*}
and since $\Lop$ is spline-admissible of order $r$ 
\begin{equation*}
\norm{(1-\chi_\Omega)\widehat{f}_j}_2^2 
\leq 
C_{\Lop} \abs{\det(\mathbf{D})}^{2jr/d-j} \norm{f}_{W_2^\Lop}^2.
\end{equation*}
Therefore
\begin{equation}\label{eq:approx2}
\abs{\det(\mathbf{D})}^{j/2}\norm{(1-\chi_\Omega)\widehat{f}_j}_2 
\leq 
C_{\Lop}^{1/2} \abs{\det(\mathbf{D})}^{jr/d} \norm{f}_{W_2^\Lop}.
\end{equation}
In order to bound the remaining term, we need a formula for the projection of $\widehat{f}_j\chi_\Omega$ onto $V_j^j$.  Notice that
\begin{align*}
1-m_j((\mathbf{D}^T)^{-j}\cdot) 
&= 
1-\frac{\abs{\widehat{\varphi}_j\circ (\mathbf{D}^T)^{-j}}^2}{\sum_{\bm{k}\in  \mathbb Z^d}\abs{\widehat{\varphi}_j\circ (\mathbf{D}^T)^{-j}(\cdot - 2\pi \bm{k})}^2}  \\
&= 
1-\frac{\abs{\widehat{\varphi_j \circ \mathbf{D}^j}}^2}{\sum_{\bm{k}\in  \mathbb Z^d}\abs{\widehat{\varphi_j \circ \mathbf{D}^j}(\cdot - 2\pi \bm{k})}^2} ,
\end{align*}
so we apply \cite[Theorem 2.20]{deboor94} to get
\begin{align*}
E(\widehat{f}_j\chi_\Omega,\widehat{V}_j^j)^2 
&=  
\int_{\Omega} \abs{\widehat{f}_j}^2 (1-m_j((\mathbf{D}^T)^{-j}\cdot) ) \\
&=  
\abs{\det(\mathbf{D})}^{-2j} \int_{\Omega}  \abs{\widehat{f} \circ (\mathbf{D}^T)^{-j}}^2 (1-m_j((\mathbf{D}^T)^{-j}\cdot) ).
\end{align*}
Now, changing variables gives
\begin{align*}
E(\widehat{f}_j\chi_\Omega,\widehat{V}_j^j)^2 
&=  
\abs{\det(\mathbf{D})}^{-j} \int_{(\mathbf{D}^T)^{-j} \Omega}  \abs{\widehat{f} }^2 \left(1+\abs{\widehat{L}}^2\right) \frac{1-m_j}{1+\abs{\widehat{L}}^2} \\
&\leq 
\abs{\det(\mathbf{D})}^{-j} \norm{f}_{W_2^\Lop}^2 \norm{ \frac{1-m_j}{1+\abs{\widehat{L}}^2}}_{L_\infty((\mathbf{D}^T)^{-j}\Omega)}. 
\end{align*}
Applying our assumption on $(1-m_j)$, we have
\begin{equation}\label{eq:approx3}
\abs{\det(\mathbf{D})}^{j/2} E(\widehat{f}_j\chi_\Omega,\widehat{V}_j^j) 
\leq 
C \abs{\det(\mathbf{D})}^{j \tilde{r}/d} \norm{f}_{W_2^\Lop}.
\end{equation}
Substituting the estimates \eqref{eq:approx2} and  \eqref{eq:approx3} into \eqref{eq:approx1} yields the result. 
\end{proof}

Concerning this theorem, an important point is that it describes the approximation properties of the MRA entirely in terms of the operator; i.e., the guaranteed approximation rates are independent of how one chooses the generalized B-splines $\varphi_j$ for the multiresolution spaces $V_j$.

\section{Operator-Like Wavelets and Riesz Bases}\label{sec:wave}

Using the non-stationary MRA defined in the previous section, we define the scale of wavelet spaces $W_j$ by the relationship
\begin{equation*}
V_{j}
=
V_{j+1} \oplus W_{j+1};
\end{equation*}
i.e., $W_{j+1}$ is the orthogonal complement of $V_{j+1}$ in $V_j$. Our goal in this section is to define Riesz bases for these spaces and for $L_2(\mathbb{R}^d)$.  To begin, let us define the functions
\begin{equation*}
\psi_{j+1}
:=
\Lop^*\phi_{j},
\end{equation*}
which we claim generate Riesz bases for the wavelet spaces, under mild conditions on the operator $\Lop$.  First, note that $\psi_{j+1}$ is indeed in $V_{j}$, because its Fourier transform $\widehat{\psi}_{j+1}$ is a periodic multiple of $\widehat{\varphi}_j$, and thus
\begin{equation}\label{eq:oplikewavfourier}
\widehat{\psi}_{j+1}(\bm{\omega})
=
\abs{\det(\mathbf{D})}^j \frac{\widehat{L}_{\mathrm{d},j}(\bm{\omega})^*}{ \sum_{\bm{k}\in \mathbb{Z}^d}\abs{\widehat{\varphi}_j(\bm{\omega}+2\pi  (\mathbf{D}^T)^{-j} \bm{k})}^2   }     \widehat{\varphi}_j(\bm{\omega}).
\end{equation}
A direct implication of our wavelet construction is the following property.
\begin{property}
The wavelet function $\psi_{j+1}$ behaves like a multiscale version of the underlying operator  $\Lop$ in the sense that, for any $f\in W_2^{\Lop}$, we have $f*\psi_{j+1}=\Lop^*(f*\phi_j)$.  Hence, in the case where $\phi_j$ is a lowpass filter, $\{\Lop^*(f*\phi_j)\}_{j\in \mathbb{Z}}$ corresponds to a multiscale representation of $\Lop^*f$.
\end{property}

The next few results focus on showing that the $\mathbf{D}^j\mathbb{Z}^d \setminus \mathbf{D}^{j+1}\mathbb{Z}^d$ shifts of $\psi_{j+1}$ are orthogonal to $V_{j+1}$ and generate a Riesz basis of $W_{j+1}$. 
\begin{proposition}\label{prop:orthog_v_w}
The wavelets $\{\psi_{j+1}(\cdot-\mathbf{D}^j\bm{k})\}_{\bm{k}\in{\mathbb{Z}^d\backslash \mathbf{D}\mathbb{Z}^d}}$ are orthogonal to the space $V_{j+1}$.
\end{proposition}
\begin{proof}
It suffices to show $\langle  \varphi_{j+1},  \psi_{j+1}(\cdot-\mathbf{D}^j\bm{k})  \rangle=0$ for every  $\bm{k} \in \mathbb{Z}^d\backslash \mathbf{D}\mathbb{Z}^d$.  From \eqref{eq:oplikewavfourier}, we have
\begin{align*}
\langle  \varphi_{j+1},  \psi_{j+1}(\cdot-\mathbf{D}^j\bm{k})  \rangle 
&= 
\int_{\mathbb{R}^d} \widehat{\varphi}_{j+1}(\bm{\omega})e^{i\bm{\omega} \cdot\mathbf{D}^j\bm{k}} \widehat{L}(\bm{\omega})\widehat{\phi}_j(\bm{\omega}) \rmd\bm{\omega} \\
&= 
\int_{\mathbb{R}^d} \widehat{L}_{\mathrm{d},j+1}(\bm{\omega})e^{i\bm{\omega} \cdot\mathbf{D}^j\bm{k}} \widehat{\phi}_j(\bm{\omega}) \rmd\bm{\omega}.
\end{align*}
Now let $\Omega$ be a fundamental domain for the lattice $2\pi (\mathbf{D}^T)^{-j}\mathbb{Z}^d$. Then
\begin{align*}
\langle  \varphi_{j+1},  \psi_{j+1}(\cdot-\mathbf{D}^j\bm{k})  \rangle 
&= 
\int_{\Omega} \widehat{L}_{\mathrm{d},j+1}(\bm{\omega})e^{i\bm{\omega} \cdot\mathbf{D}^j\bm{k}} \sum_{\bm{n} \in \mathbb{Z}^d}\widehat{\phi}_j(\bm{\omega}-2\pi (\mathbf{D}^T)^{-j} \bm{n}) \rmd\bm{\omega} \\
&= 
\abs{\det(\mathbf{D})}^j \int_{\Omega} \widehat{L}_{\mathrm{d},j+1}(\bm{\omega})e^{i\bm{\omega} \cdot\mathbf{D}^j\bm{k}}  \rmd\bm{\omega}.
\end{align*}
From Definition \ref{def:spadm}, we know that $\widehat{L}_{\mathrm{d},j+1}$ has a series representation of the form $\sum_{\bm{n} \in \mathbb{Z}^d} p_{j+1}[\bm{n}]e^{i\bm{\omega}\cdot \mathbf{D}^{j+1} \bm{n}}$, so
\begin{align*}
\langle  \varphi_{j+1},  \psi_{j+1}(\cdot-\mathbf{D}^j\bm{k})  \rangle 
&= 
\abs{\det(\mathbf{D})}^j \sum_{\bm{n} \in \mathbb{Z}^d} p_{j+1}[\bm{n}] \int_{\Omega} e^{i\bm{\omega}\cdot \mathbf{D}^j  \left(\mathbf{D}\bm{n} + \bm{k} \right)} \rmd\bm{\omega}\\
&= 
\sum_{\bm{n} \in \mathbb{Z}^d} p_{j+1}[\bm{n}] \int_{[0,2\pi]^d} e^{i\bm{\omega}\cdot   \left(\mathbf{D}\bm{n} + \bm{k} \right)} \rmd\bm{\omega}.
\end{align*}
Since $\bm{k} \notin \mathbf{D} \mathbb{Z}^d$, we see that $\mathbf{D}\bm{n} + \bm{k} \neq 0$ for any $\bm{n}$, and this implies that $\langle  \varphi_{j+1},  \psi_{j+1}(\cdot-\mathbf{D}^j\bm{k})  \rangle=0$. 
\end{proof}

In order to prove that the $\mathbf{D}^j\mathbb{Z}^d \setminus \mathbf{D}^{j+1}\mathbb{Z}^d$ shifts of $\psi_{j+1}$ form a Riesz basis of the wavelet space $W_{j+1}$, we introduce notation that will help us formulate the problem as a shift-invariant one.  In the following definition, we use the fact that there is a set of vectors
\begin{equation*}
\left\{\bm{e}_{l}\in \mathbb{Z}^d: l=0,..,\abs{\det(\mathbf{D})}-1\right\}
\end{equation*} 
such that
\begin{equation*}
\bigcup_{l=0}^{\abs{\det(\mathbf{D})}-1} \left(\mathbf{D}^j\bm{e}_{l} + \mathbf{D}^{j+1}\mathbb{Z}^d\right) 
= 
\mathbf{D}^{j} \mathbb{Z}^d.
\end{equation*}
\begin{definition}
For every $j\in\mathbb{Z}$ and every $l\in \{1,\dots,\abs{\det(\mathbf{D})}-1\}$, we define the 
wavelets 
\begin{equation*}
\psi_{j+1}^{(l)} (\bm{x})
:= 
\psi_{j+1} (\bm{x}-\mathbf{D}^j\bm{e}_{l}),
\end{equation*}
and we define the collections
\begin{equation*}
\Psi
:=
\Psi_{j+1}
:=
\left\{\psi^{(l)}_{j+1}\right\}_{l=1}^{\abs{\det(\mathbf{D})}-1}.
\end{equation*}
\end{definition}
In the following, necessary and sufficient conditions on the operator $\Lop$ are given which guarantee that $\Psi_{j+1}$ generates a Riesz basis of $W_{j+1}$.  The technique used is called fiberization, and it can be applied to characterize finitely generated shift-invariant spaces \cite{ron95}.
In this setting, a collection of functions defines a Gramian matrix, and the property of being a Riesz basis is equivalent to the Gramian having bounded eigenvalues.  In our situation, the Gramian for $\Psi$ is 
\begin{align*}
G_{\Psi}(\bm{\omega})
&= 
\abs{\det(\mathbf{D})}^{-j-1} \left( \sum_{\bm{\beta} \in 2\pi (\mathbf{D}^T)^{-j-1}\mathbb{Z}^d}   \widehat{\psi^{(k)}_{j+1}}\left(\bm{\omega}+\bm{\beta} \right) \widehat{\psi^{(l)}_{j+1}}\left(\bm{\omega}+\bm{\beta} \right)^* \right)_{k,l}\\ 
&= 
\abs{\det(\mathbf{D})}^{-j-1} \left( \sum_{\bm{\beta} \in 2\pi (\mathbf{D}^T)^{-j-1}\mathbb{Z}^d} e^{- i \mathbf{D}^{j}(\bm{e}_k-\bm{e}_l)\cdot (\bm{\omega}+\bm{\beta})} \abs{\widehat{\psi}_{j+1}(\bm{\omega}+\bm{\beta})}^2  \right)_{k,l},
\end{align*}
where $k$ and $l$ range from $1$ to $\abs{\det(\mathbf{D})}-1$. The normalization factor $\abs{\det(\mathbf{D})}^{-j-1}$ accounts for scaling of the lattice.

Let us denote the largest and smallest eigenvalues of $G_{\Psi}(\bm{\omega})$ by $\Lambda(\bm{\omega})$ and $\lambda(\bm{\omega})$, respectively. Then the collection $\Psi$ generates a Riesz basis if and only if $\Lambda$ and $1/\lambda$ are essentially bounded (cf. \cite{ron95} Theorem 2.3.6).  To simplify this matrix without changing the eigenvalues, we apply the similarity transformation $T(\bm{\omega})^{-1}G_{\Psi}(\bm{\omega})T(\bm{\omega})$, where $T$ is the square diagonal matrix with diagonal entry $e^{-i\mathbf{D}^{j}\bm{e}_l\cdot\bm{\omega}}$ in row $l$.  This transformation multiplies column $l$ of $G_\Psi$ by $e^{-i\mathbf{D}^{j}\bm{e}_l\cdot\bm{\omega}}$ and row $k$ of $G_\Psi$ by $e^{i\mathbf{D}^{j}\bm{e}_k\cdot\bm{\omega}}$. Since the eigenvalues are unchanged, let us call this new matrix $G_\Psi$ as well. We then have
\begin{equation*}
G_\Psi (\bm{\omega}) 
=  
\abs{\det(\mathbf{D})}^{-j-1}\left(  \sum_{\bm{\beta}\in 2\pi (\mathbf{D}^T)^{-j-1}\mathbb{Z}^d} e^{- i \mathbf{D}^{j}(\bm{e}_k-\bm{e}_l)\cdot \bm{\beta}} \abs{\widehat{\psi}_{j+1}(\bm{\omega}+\bm{\beta})}^2  \right)_{k,l}.
\end{equation*}
Using the fact that $\bigcup_m \left(\bm{e}_m+\mathbf{D}^T\mathbb{Z}^d\right)=\mathbb{Z}^d$ and the notation
\begin{equation}\label{eq:rieszbds_cmw}
c(m;\bm{\omega})
:=  
\abs{\det(\mathbf{D})}^{-j-1}\sum_{\bm{\beta}\in 2\pi (\mathbf{D}^T)^{-j} \mathbb{Z}^d}  \abs{\widehat{\psi}_{j+1}(\bm{\omega}+2\pi  (\mathbf{D}^T)^{-j-1} \bm{e}_m+\bm{\beta})}^2,
\end{equation}
we write
\begin{equation*}
G_\Psi(\bm{\omega}) 
=  
\left( \sum_{m=0}^{\abs{\det(\mathbf{D})}-1} c(m;\bm{\omega})  e^{-2\pi i (\bm{e}_k-\bm{e}_l)\cdot (\mathbf{D}^T)^{-1}\bm{e}_m} \right)_{k,l}.
\end{equation*}
\begin{definition}
Let $\mathbf{H}$ be the $\abs{\det(\mathbf{D})} \times \abs{\det(\mathbf{D})}$ matrix 
\begin{equation*}
\mathbf{H} 
:= 
\abs{\det(\mathbf{D})}^{-1/2} \left( e^{2\pi i\bm{e}_m\cdot (\mathbf{D}^T)^{-1}\bm{e}_k} \right)_{k,m},
\end{equation*}
which is the complex conjugate of the discrete Fourier transform matrix for the lattice generated by $\mathbf{D}^T$ \cite{vaidyanathan90}. 
Here,  $k$ and $m$ range over the index set $\mathcal{M}:=\{0,\dots,\abs{\det(\mathbf{D})}-1\}$. Also, define $\mathbf{H}_0$  to be the submatrix obtained by removing column $0$ from $\mathbf{H}$. 
\end{definition}
\begin{lemma}\label{lem:eig_values}
The minimum and maximum eigenvalues $\lambda(\bm{\omega}),\Lambda(\bm{\omega})$ of the Gramian matrix $G_{\Psi}(\bm{\omega})$ satisfy the following properties:
\begin{enumerate}[(i)]
\item 
${\displaystyle  \lambda(\bm{\omega}) \geq \abs{\det(\mathbf{D})} \min_{m\in\mathcal{M}} c(m;\bm{\omega}) \quad \text{and} \quad \Lambda(\bm{\omega})  \leq \abs{\det(\mathbf{D})} \max_{m\in\mathcal{M}} c(m;\bm{\omega}) }$
\item 
There is a constant $C>0$ such that 
\begin{equation*}
\lambda(\bm{\omega}) 
\geq 
C \abs{\det(\mathbf{D})} \max_{m_0(\bm{\omega}) \in\mathcal{M}  }\min_{ m\in \mathcal{M} \setminus \{m_0(\bm{\omega})\} } c(m;\bm{\omega}).
\end{equation*}
\item If for any fixed $\bm{\omega}\in\mathbb{R}^d$, there exist distinct $m_1(\bm{\omega}),m_2(\bm{\omega})\in \mathcal{M}$ such that $c(m_1(\bm{\omega});\bm{\omega})=c(m_2(\bm{\omega});\bm{\omega})=0$, then $\lambda(\bm{\omega})=0$.
\end{enumerate}
\end{lemma}
\begin{proof}
The Gramian matrix $G_{\Psi}$ can be written as 
\begin{equation*}
G_\Psi(\bm{\omega}) 
= 
\abs{\det(\mathbf{D})} \mathbf{H}_0^* \mathcal{D}(\bm{\omega}) \mathbf{H}_0,
\end{equation*}
where $\mathcal{D}(\bm{\omega})$ is the  $\abs{\det(\mathbf{D})} \times \abs{\det(\mathbf{D})}$ diagonal matrix with entry $c(m;\bm{\omega})$ in column $m$:
\begin{align*}
\abs{\det(\mathbf{D})} \mathbf{H}_0^* \mathcal{D}(\bm{\omega}) \mathbf{H}_0 
&= 
\left( e^{-2\pi i\bm{e}_k\cdot (\mathbf{D}^T)^{-1}\bm{e}_n} \right)_{k,n} 
\left(c(m;\bm{\omega})\right)_{n,m} 
\left( e^{2\pi i\bm{e}_l\cdot (\mathbf{D}^T)^{-1}\bm{e}_m} \right)_{m,l} \\
&= 
\left( c(m;\bm{\omega})e^{-2\pi i\bm{e}_k\cdot (\mathbf{D}^T)^{-1}\bm{e}_m} \right)_{k,m} 
\left( e^{2\pi i\bm{e}_l\cdot (\mathbf{D}^T)^{-1}\bm{e}_m} \right)_{m,l} \\
&= 
\left(\sum_{m=0}^{\abs{\det(\mathbf{D})}-1}  c(m;\bm{\omega})e^{-2\pi i (\bm{e}_k -\bm{e}_l )\cdot (\mathbf{D}^T)^{-1}\bm{e}_m}  \right)_{k,l}.
\end{align*}
Since $\mathcal{D}(\bm{\omega})$ has non-negative entries, we write this as
\begin{equation*}
G_\Psi(\bm{\omega}) 
= 
\abs{\det(\mathbf{D})} (\mathcal{D}(\bm{\omega})^{1/2} \mathbf{H}_0)^* (\mathcal{D}(\bm{\omega})^{1/2} \mathbf{H}_0).
\end{equation*}

Now consider the quadratic form
\begin{align*}
\bm{\alpha}^*G_\Psi(\bm{\omega})\bm{\alpha} 
&= 
\abs{\det(\mathbf{D})} (\mathcal{D}(\bm{\omega})^{1/2} \mathbf{H}_0\bm{\alpha})^* (\mathcal{D}(\bm{\omega})^{1/2} \mathbf{H}_0\bm{\alpha})\\
&= 
\abs{\det(\mathbf{D})} \abs{\mathcal{D}(\bm{\omega})^{1/2} \mathbf{H}_0\bm{\alpha}}^2,
\end{align*}
where $\bm{\alpha}\in\mathbb{C}^{\abs{\det(\mathbf{D})}-1}$.
Since $\mathbf{H}_0$ is an isometry, $\abs{\mathbf{H}_0 \bm{\alpha}}=\abs{\bm{\alpha}}$, and we immediately verify {\it (i)}.

To prove {\it (ii)}, we first identify the range of $\mathbf{H}_0$.  By the Fredholm alternative, a vector is in the range of $\mathbf{H}_0$ if and only if it is orthogonal to the null space of $\mathbf{H}_0^*$.  Since $\mathbf{H}^*$ is a unitary matrix and its first row is a constant multiple of $(1,1,\dots,1)^T$, the range of $\mathbf{H}_0$ consists of vectors that are orthogonal to $(1,1,\dots,1)^T$. Therefore 
\begin{align*}
\lambda(\bm{\omega}) 
& = 
\abs{\det(\mathbf{D})}  \min_{\substack{\bm{\alpha}\in\mathbb{C}^{\abs{\det(\mathbf{D})}-1} \\ \abs{\bm{\alpha}}=1}} \abs{\mathcal{D}(\bm{\omega})^{1/2} \mathbf{H}_0\bm{\alpha}}^2\\
& =
\abs{\det(\mathbf{D})} \min_{\substack{\bm{\alpha}\in\mathbb{C}^{\abs{\det(\mathbf{D})}} \\ \abs{\bm{\alpha}}=1 \\ \bm{\alpha}\perp (1,1,\dots,1)^T}} \abs{\mathcal{D}(\bm{\omega})^{1/2} \bm{\alpha}}^2 \\
& = 
\abs{\det(\mathbf{D})} \min_{\substack{\bm{\alpha}\in\mathbb{C}^{\abs{\det(\mathbf{D})}} \\ \abs{\bm{\alpha}}=1 \\ \bm{\alpha}\perp (1,1,\dots,1)^T}} \sum_{m\in \mathcal{M}} \abs{\alpha_m}^2 c(m;\bm{\omega}),
\end{align*}
where in the last equation, we use the notation $\bm{\alpha} = (\alpha_0,\dots,\alpha_{\abs{\det(\mathbf{D})}-1})$.
Then a lower bound is given by
\begin{align*}
\lambda(\bm{\omega}) 
&\geq 
\abs{\det(\mathbf{D})} \max_{m_0(\bm{\omega}) \in\mathcal{M} } \min_{\substack{\bm{\alpha}\in\mathbb{C}^{\abs{\det(\mathbf{D})}} \\ \abs{\bm{\alpha}}=1 \\ \bm{\alpha}\perp (1,1,\dots,1)^T}} \sum_{m\in \mathcal{M}\setminus \{m_0(\bm{\omega})\} } \abs{\alpha_m}^2 c(m;\bm{\omega}) \\
&\geq 
\abs{\det(\mathbf{D})} \max_{m_0(\bm{\omega}) \in\mathcal{M} } \min_{\substack{\bm{\alpha}\in\mathbb{C}^{\abs{\det(\mathbf{D})}} \\ \abs{\bm{\alpha}}=1 \\ \bm{\alpha}\perp (1,1,\dots,1)^T}} 
\left( \min_{m\in \mathcal{M} \setminus \{m_0(\bm{\omega})\}} c(m;\bm{\omega})\right)
\sum_{m\in \mathcal{M}\setminus \{m_0(\bm{\omega})\} } \abs{\alpha_m}^2 \\
& = 
\abs{\det(\mathbf{D})} \max_{m_0(\bm{\omega}) \in\mathcal{M} }
\left(\min_{m\in \mathcal{M} \setminus \{m_0(\bm{\omega})\}} c(m;\bm{\omega})\right)
\left( \min_{\substack{\bm{\alpha}\in\mathbb{C}^{\abs{\det(\mathbf{D})}} \\ \abs{\bm{\alpha}}=1 \\ \bm{\alpha}\perp (1,1,\dots,1)^T}} \sum_{m\in \mathcal{M}\setminus \{m_0(\bm{\omega})\} } \abs{\alpha_m}^2 \right).
\end{align*}

Notice that none of the standard unit vectors
\begin{equation*}
\{(1,0,\dots,0),(0,1,0,\dots,0),\dots,(0,\dots,0,1)\}
\end{equation*} 
are orthogonal to $(1,1,\dots,1)^T$, so there is a constant $C>0$ such that 
\begin{equation*}
\min_{m_0(\bm{\omega})\in \mathcal{M}}\min_{\substack{\bm{\alpha}\in\mathbb{C}^{\abs{\det(\mathbf{D})}} \\ \abs{\bm{\alpha}}=1 \\ \bm{\alpha}\perp (1,1,\dots,1)^T}} \sum_{m\in\mathcal{M}\setminus \{m_0(\bm{\omega})\}}\abs{\alpha_m}^2 
=
C.
\end{equation*}
We now use this constant to provide a lower bound for $\lambda(\bm{\omega})$: 
\begin{equation*}
\lambda(\bm{\omega}) 
\geq C 
\abs{\det(\mathbf{D})} \max_{m_0(\bm{\omega}) \in\mathcal{M} }\min_{m\in \mathcal{M} \setminus \{m_0(\bm{\omega})\}} c(m;\bm{\omega}).
\end{equation*}

Finally, for {\it (iii)}, suppose that there are distinct $m_1(\bm{\omega}),m_2(\bm{\omega})\in \mathcal{M}$ such that $c(m_1(\bm{\omega});\bm{\omega})=c(m_2(\bm{\omega});\bm{\omega})=0$.  Then define the vector $\bm{\alpha}=(\alpha_1,\dots,\alpha_{\abs{\det(\mathbf{D})}-1})\in \mathbb{C}^{\abs{\det(\mathbf{D})}}$ such that $\alpha_{m_1}=1/\sqrt{2}$, $\alpha_{m_2}=-1/\sqrt{2}$, and all other entries are zero. This vector is in the range of $\mathbf{H}_0$, and $\mathcal{D}(\bm{\omega})^{1/2} \bm{\alpha}=0$. Therefore
$\lambda(\bm{\omega})=0$. 
\end{proof}
\begin{lemma}\label{lem:rb}
The collection $\Psi$ generates a Riesz basis if and only if no two of the functions $c(m;\bm{\omega})$ are zero for the same $\bm{\omega}$. 
\end{lemma}
\begin{proof}
 Let $\Omega_j$ be a fundamental domain for the lattice $2\pi (\mathbf{D}^T)^{-j}\mathbb{Z}^d$, and let $\overline{\Omega}_j$ denote its closure.

For the reverse direction, we must show that there is a uniform lower bound of $\lambda(\bm{\omega})$ over $\overline{\Omega}_j$.  
By Lemma \ref{lem:eig_values}, it suffices to provide a lower bound for
\begin{equation}\label{eq:riesz_uniform_fd}
\max_{m_0(\bm{\omega}) \in \mathcal{M}}  \min_{m\in \mathcal{M}\setminus \{m_0(\bm{\omega})\}} c(m;\bm{\omega}).
\end{equation}
Based on \eqref{eq:oplikewavfourier} and \eqref{eq:rieszbds_cmw}, we verify that  
\begin{equation}\label{eq:coeff_fcn_rb}
c(m;\bm{\omega}) 
= 
\abs{\det(\mathbf{D})}^{j-1}  \frac{\abs{ \widehat{L}_{\mathrm{d},j}(\bm{\omega}+2\pi  (\mathbf{D}^T)^{-j-1} \bm{e}_m)}^2}{\sum_{\bm{k}\in \mathbb{Z}^d} \abs{\widehat{\varphi}_j(\bm{\omega}+2\pi  (\mathbf{D}^T)^{-j-1} \bm{e}_m+2\pi  (\mathbf{D}^T)^{-j} \bm{k})}^2}.
\end{equation}
Note that the numerator is a continuous function, and the denominator is bounded away from zero, due to the Riesz basis condition on $\varphi_j$. Hence, $c(m;\bm{\omega})=0$ at a point $\bm{\omega}$ if and only if  $\widehat{L}_{\mathrm{d},j}(\bm{\omega}+2\pi  (\mathbf{D}^T)^{-j-1} \bm{e}_m)=0$.
Let us define the continuous function
\begin{equation}\label{eq:cont_lower_bnd}
F(\bm{\omega}) 
:= 
\max_{m_0(\bm{\omega}) \in \mathcal{M}}  \min_{m\in \mathcal{M}\setminus \{m_0(\bm{\omega})\}} \abs{ \widehat{L}_{\mathrm{d},j}(\bm{\omega}+2\pi  (\mathbf{D}^T)^{-j-1} \bm{e}_m)}^2.
\end{equation}
Since no two functions $c(m;\bm{\omega})$ are zero at any point $\bm{\omega}$, $F$ is positive on $\overline{\Omega}_j$. Due to the compactness of this set, there is a constant $C>0$ such that $F(\bm{\omega})>C$ on $\overline{\Omega}_j$. Since $F$ is bounded away from zero,  \eqref{eq:riesz_uniform_fd} is as well.

The forward direction follows from {\it (iii)} of Lemma \ref{lem:eig_values}. 
\end{proof}
\begin{lemma}\label{lem:rb_w}
If $\Psi_{j+1}$ generates a Riesz basis, then it provides a Riesz basis for  $W_{j+1}$.
\end{lemma}
\begin{proof}
We verify this fact by comparing the bases
\begin{align*}
\Psi_{j+1}' 
& :=
\{\varphi_{j+1}\} \bigcup \Psi_{j+1}, \\
\Phi_{j} 
& :=
\{\varphi_{j}(\cdot-\mathbf{D}^{j}\bm{e}_m)\}_{m\in\mathcal{M}}
\end{align*}
for shift invariant spaces on the lattice $\mathbf{D}^{j+1}\mathbb{Z}^d$.
The $\mathbf{D}^{j}\mathbb{Z}^d$ shifts of $\varphi_j$ are a Riesz basis for $V_j$, or, equivalently, the $\mathbf{D}^{j+1}\mathbb{Z}^d$ shifts of the elements of $\Phi_{j}$ are a Riesz basis of $V_j$. This basis has $\abs{\det(\mathbf{D})}$ elements, and any other basis must have the same number of elements.

The collections $\Psi_{j+1}$ and $\{\varphi_{j+1}\}$ generate Riesz bases, and both are contained in $V_j$. These bases are orthogonal, as was shown in Proposition \ref{prop:orthog_v_w}. Therefore $\Psi_{j+1}'$ generates a Riesz basis for a subspace of $V_{j}$, and $\Psi_{j+1}$ generates a Riesz basis for a subspace of $W_{j+1}$. The fact that $\Psi_{j+1}'$ has $\abs{\det(\mathbf{D})}$ elements implies that $\Psi_{j+1}'$ provides a Riesz basis for $V_j$ (cf. \cite[Theorem 2.26]{deboor93} and \cite{aldroubi95}), and the result follows. 
\end{proof}

In Lemmas \ref{lem:eig_values} and \ref{lem:rb}, we saw how $\Psi$ generating a Riesz basis depends on the zeros of the functions $c(m;\cdot)$.  From \eqref{eq:coeff_fcn_rb}, it is clear that the zeros of $c(m;\cdot)$ coincide with the zeros of a shifted version of 
$\widehat{L}_{\mathrm{d},j}$.  In order to interpret the Riesz basis conditions in terms of the operator $\Lop$, we note that the zeros of $\widehat{L}_{\mathrm{d},j}$ are precisely the periodized zeros of $\widehat{L}$.  

Let us denote the zero set of the symbol $\widehat{L}$ as 
\begin{equation*}
\mathcal N 
:=
\{\bm{p}\in{\mathbb R^d} : \widehat{L}(\bm{p})=0\},
\end{equation*}
and for each scale $j$ and each $m=0,\dots, \abs{\det(\mathbf{D})}-1$, let us define the periodized sets
\begin{equation*}
\mathcal{N}_j^{(m)} 
:= 
\left\{ \bm{p}-2\pi (\mathbf{D}^T)^{-j-1} \bm{e}_m+2\pi (\mathbf{D}^T)^{-j}\bm{k}  : \bm{p}\in \mathcal{N}, \bm{k}\in \mathbb{Z}^d \right\}.
\end{equation*}
Note that $\mathcal{N}_j^{(m)}$ is the zero set of $\widehat{L}_{\mathrm{d},j}(\cdot + 2\pi (\mathbf{D}^T)^{-j-1} \bm{e}_m )$, and hence it is also the zero set of $c(m,\bm{\omega})$. 
\begin{theorem}\label{th:rb}
Let $j\in \mathbb{Z}$ be an arbitrary scale. Then the family of functions 
\begin{equation*}
\Psi_{j+1}
=
\left\{\psi^{(m)}_{j+1}\right\}_{m=1}^{\abs{\det(\mathbf{D})}-1}
\end{equation*} 
generates a Riesz basis of $W_{j+1}$ if and only if the sets $\mathcal N_j^{(m)}$ satisfy
\begin{equation}\label{eq:suffcond}
\mathcal N_j^{(0)}\cap\mathcal N_j^{(m)}
=
\emptyset
\end{equation}
for each $1\leq m\leq \abs{\det(\mathbf{D})}-1$.
\end{theorem}

Our wavelet construction is intended to be general so that we may account for a large collection of operators.  As a consequence of this generality, we cannot conclude that our wavelet construction always produces a Riesz basis of $L_2(\mathbb{R}^d)$. Here, we impose additional conditions on $\Lop$ and multiply $\psi_{j+1}$ by an appropriate normalization factor to ensure that a Riesz basis is produced. In order to preserve generality, we focus on the fine scale wavelet spaces and include a multiresolution space $V_{j_0+1}$ in our Riesz basis.
\begin{theorem}\label{th:rb2}
Let $\Lop$ be a spline admissible operator 
of order $r$, and suppose that there exist $\omega_0>0$ and constants $C_1,C_2>0$ such that the symbol $\widehat{L}$ satisfies 
\begin{equation*}
C_1 \abs{\bm{\omega}}^{r} 
\leq 
\abs{\widehat{L}(\bm{\omega})} 
\leq 
C_2 \abs{\bm{\omega}}^r
\end{equation*}
for $\abs{\bm{\omega}} \geq \omega_0$.  Then there is an integer $j_0$ such that the collection
\begin{equation}\label{eq:rb}
\left\{\varphi_{j_0+1}(\cdot-\beta)\right\}_{\beta \in \mathbf{D}^{j_0+1}\mathbb{Z}^d}   \bigcup_{j\leq j_0} 
\left\{ \abs{\det(\mathbf{D})}^{(r/d-1/2)j}\psi_{j+1}(\cdot-\beta)\right\}_{\beta \in \mathbf{D}^{j}\mathbb{Z}^d\backslash\mathbf{D}^{j+1}\mathbb{Z}^d }
\end{equation} 
forms a Riesz basis of $L_2(\mathbb{R}^d)$. 
\end{theorem}
\begin{proof}
Let $j_0$ be an integer for which $\omega_0<\pi/4 \abs{\det(\mathbf{D})}^{-j_0/d}$. Considering Lemma \ref{lem:eig_values}, the Riesz bounds for the wavelet spaces depend on the functions $c(m;\bm{\omega})$ of \eqref{eq:rieszbds_cmw}. A Fourier domain formula for the wavelet $\psi_{j+1}$ is
\begin{equation*}
\widehat{\psi}_{j+1} 
= 
\abs{ \text{det} (\mathbf{D}) }^j \frac{\widehat{L}(\bm{\omega})^{-1}}{\sum_{\bm{k}\in\mathbb{Z}^d} \abs{\widehat{L}(\bm{\omega}+2\pi (\mathbf{D}^T)^{-j}\bm{k})}^{-2}},
\end{equation*}
so we have
\begin{align*}
c(m;\bm{\omega}) 
&= 
\abs{\det(\mathbf{D})}^{-j-1}\sum_{\bm{\beta}\in 2\pi (\mathbf{D}^T)^{-j} \mathbb{Z}^d}  \abs{\widehat{\psi}_{j+1}(\bm{\omega}+2\pi  (\mathbf{D}^T)^{-j-1} \bm{e}_m+\bm{\beta})}^2\\
&= 
\abs{ \text{det} (\mathbf{D}) }^{j-1} \sum_{\bm{\beta}\in 2\pi (\mathbf{D}^T)^{-j} \mathbb{Z}^d} \abs{ \frac{\widehat{L}(\bm{\omega}+2\pi (\mathbf{D}^T)^{-j-1}\bm{e}_m +\bm{\beta})^{-1}}{\sum_{\bm{k}\in\mathbb{Z}^d} \abs{\widehat{L}(\bm{\omega}+2\pi (\mathbf{D}^T)^{-j-1}\bm{e}_m +2\pi (\mathbf{D}^T)^{-j}\bm{k})}^{-2}} }^2.
\end{align*}
We now need upper and lower bounds on the terms
\begin{equation}\label{eq:psi_sum}
\sum_{\bm{\beta}\in 2\pi (\mathbf{D}^T)^{-j} \mathbb{Z}^d}  \abs{\frac{\widehat{L}(\bm{\omega}+2\pi  (\mathbf{D}^T)^{-j-1} \bm{e}_m+\bm{\beta})^{-1}}{\sum_{\bm{k} \in \mathbb{Z}^d}\abs{\widehat{L}(\bm{\omega}+2\pi  (\mathbf{D}^T)^{-j-1} \bm{e}_m+2\pi (\mathbf{D}^T)^{-j}\bm{k})}^{-2}}}^2.
\end{equation}
Recall that the upper bound should be uniform across all values of $m$; however, for the lower bound, it is sufficient to consider only $\abs{\det(\mathbf{D})}-1$ of the functions $c(m;\bm{\omega})$.

Define the lattices
\begin{equation*}
X_j(m,\bm{\omega})
:=
\left\{ \bm{x}_{\bm{k}}=\bm{\omega}+2\pi  (\mathbf{D}^T)^{-j-1} \bm{e}_m+2\pi (\mathbf{D}^T)^{-j}\bm{k}: \bm{k}\in\mathbb{Z}^d \right\}.
\end{equation*}
The value of \eqref{eq:psi_sum} depends on position of $X_j(m,\bm{\omega})$ with respect to the origin, as well as two density parameters. Let us introduce the notation $h_j$ for the fill distance and $q_j$ for the separation radius of $X_j(m,\bm{\omega})$.  Since each lattice $X_j(m,\bm{\omega})$ is a translation of $X_j(0,\bm{0})$, the quantities $h_j$ and $q_j$ are independent of $m$ and $\bm{\omega}$, and they are defined as
\begin{align*}
h_j 
&:= 
\sup_{\bm{y} \in \mathbb{R}^d} \inf_{\bm{x}\in X_j(0,\bm{0})} \abs{\bm{y}-\bm{x}}\\
q_j 
&:= 
\frac{1}{2}\inf_{ \substack{\bm{x}, \bm{x}' \in X_j(0,\bm{0}) \\ \bm{x}\neq \bm{x}' }} \abs{\bm{x}-\bm{x}'}.
\end{align*}
Given the structure of the matrix $\mathbf{D}$, we can compute
\begin{align*}
h_j 
&= 
2\pi \abs{\det(\mathbf{D})}^{-j/d} \sup_{\bm{y} \in \mathbb{R}^d} \inf_{\bm{k}\in\mathbb{Z}^d} \abs{ \bm{y} - \bm{k} } \\
&=  
\pi \abs{\det(\mathbf{D})}^{-j/d} \sqrt{d},
\end{align*}
and likewise
\begin{equation*}
q_j 
= 
\pi \abs{\det(\mathbf{D})}^{-j/d}.
\end{equation*}

Considering the distance function
\begin{equation*}
\text{dist}(\bm{0},X_j(m,\bm{\omega})) 
:= 
\min_{\bm{x} \in X_j(m,\bm{\omega})} \abs{\bm{x}},
\end{equation*}
we bound \eqref{eq:psi_sum} by considering two cases:
\begin{enumerate}
\item $\text{dist}(\bm{0},X_j(m,\bm{\omega})) \geq q_j/2$;
\item $\text{dist}(\bm{0},X_j(m,\bm{\omega})) < q_j/2$.
\end{enumerate} 
For Case 1, all points of the lattice $X_j(m,\bm{\omega})$ lie outside of the ball of radius $\omega_0$ centered at the origin.  Therefore, \eqref{eq:psi_sum} can be reduced to
\begin{equation}\label{eq:psi_sum_reduced}
\left(    \sum_{\bm{x}_{\bm{k}} \in X_j(m,\bm{\omega})}  \abs{\widehat{L}(\bm{x}_{\bm{k}})}^{-2}      \right)^{-1}.
\end{equation}
Applying Proposition \ref{pr:lattice_lower}, we can bound \eqref{eq:psi_sum_reduced} from above by a constant multiple of $h_j^{2r}=(\pi\sqrt{d})^{2r} \abs{\det(\mathbf{D})}^{-2rj/d}$, and applying Proposition \ref{pr:lattice_upper}, we bound \eqref{eq:psi_sum_reduced} from below by a constant multiple of $q_j^{2r}=\pi^{2r}\abs{\det(\mathbf{D})}^{-2rj/d}$.  Importantly, the proportionality constants are independent of $j$.

For Case 2, we must be more careful, as one of the lattice points lies close to the origin. However, for any fixed $\bm{\omega}$, there is at most one $m$ for which $\text{dist}(\bm{0},X_j(m,\bm{\omega})) < q_j/2$. Therefore, in this case, a sufficient lower bound for \eqref{eq:psi_sum} is $0$; however, the upper bound must match the one derived in Case 1. Let us further separate Case 2 into the cases
\begin{enumerate}[2a.]
\item $\widehat{L}$ takes the value $0$ at some point of the lattice $X_j(m,\bm{\omega})$
\item $\absf{\widehat{L}^{-1}}<\infty$ for every point of the lattice $X_j(m,\bm{\omega})$
\end{enumerate}
In Case 2a, we see that \eqref{eq:psi_sum} is $0$. In Case 2b, we again reduce \eqref{eq:psi_sum} to \eqref{eq:psi_sum_reduced}, and we see that $0$ is a lower bound for \eqref{eq:psi_sum_reduced}.  For the upper bound, we apply Proposition \ref{pr:lattice_lower}, and the bound coincides with the one obtained in Case 1. 

To finish the proof, we note that Lemmas \ref{lem:eig_values}, \ref{lem:rb}, and \ref{lem:rb_w} imply that the wavelets 
\begin{equation*}
\left\{ \psi_{j+1}(\cdot-\beta)\right\}_{\beta \in \mathbf{D}^{j}\mathbb{Z}^d\backslash\mathbf{D}^{j+1}\mathbb{Z}^d }
\end{equation*}
form a Riesz basis at each level $j\leq j_0$. Furthermore, the bounds obtained here on $c(m,\bm{\omega})$ imply that the Riesz bounds are proportional to $\abs{\det(\mathbf{D})}^{(1/2-r/d)2j}$.  Therefore, the collection \eqref{eq:rb} is a Riesz basis of $L_2(\mathbb{R}^d)$. 
\end{proof}

Let us remark that in the proof of this theorem, we used the fact that the lattices corresponding to the matrix $\mathbf{D}$ scale uniformly in all directions.  This allowed us to find upper and lower Riesz bounds that are independent of $j$.  However, the Riesz bounds would depend on $j$ for general integer dilation matrices.

\section{Decorrelation of Coefficients}\label{sec:decor}

As was stated in the introduction, the primary reason for our construction is to promote a sparse wavelet representation.  Our model is based on the assumption that the wavelet coefficients of a signal $s$ are computed by the $L_2$ inner product
\begin{equation*}
\left<s,\psi_j(\cdot-\mathbf{D}^j\bm{k})\right>.
\end{equation*}
Here, we should point out that, unless the wavelets form an orthogonal basis, reconstruction will be defined in terms of a dual basis.  However, as our focus in this paper is the sparsity of the coefficients,  we are content to work with the analysis component of the approximation and leave the synthesis component for future study.

Now, considering our stochastic model, it is important to use wavelets that (nearly) decorrelate the signal within each scale, and one way to accomplish this goal is by modifying the underlying operator. 
Hence, given a spline-admissible pair, $\Lop$ and $\mathbf{D}$, we define a new spline-admissible pair, $\Lop_n$ and $\mathbf{D}$, by $\widehat{L}_{n}:=\widehat{L}^n$, and we shall see that as $n$ increases, the wavelet coefficients become decorrelated.  This result follows from the fact that the $(\Lop_n)^*\Lop_n$-spline interpolants (appropriately scaled) converge to a sinc-type function, and it is motivated by the work of Aldroubi and Unser, which shows that a large family of spline-like interpolators converge to the ideal sinc interpolator \cite{unser94}. To state this result explicitly, we denote the generalized B-splines for $\Lop_n$ by $\widehat{\varphi}_{n,j}=\widehat{\varphi}_{ j}^n $. Therefore the $(\Lop_n)^*\Lop_n$-spline interpolants are given by
\begin{align*}
\widehat{\phi}_{n,j}(\bm{\omega})
&=
\abs{\det(\mathbf{D})}^{j}\frac{\abs{\widehat{\varphi}_{n,j}(\bm{\omega})}^2}{\sum_{\bm{k} \in \mathbb{Z}^d}  \abs{\widehat{\varphi}_{n,j}(\bm{\omega}+ 2\pi (\mathbf{D}^T)^{-j} \bm{k})}^2   } \\
&= 
\abs{\det(\mathbf{D})}^{j}\frac{\abs{\widehat{\varphi}_{j}(\bm{\omega})}^{2n}}{\sum_{\bm{k} \in \mathbb{Z}^d}  \abs{\widehat{\varphi}_{j}(\bm{\omega}+ 2\pi (\mathbf{D}^T)^{-j} \bm{k})}^{2n}   },
\end{align*}
and we analogously define
\begin{equation*}
m_{n,j}(\bm{\omega})
=  
\frac{\abs{\widehat{\varphi}_{j}(\bm{\omega})}^{2n}}{\sum_{\bm{k} \in \mathbb{Z}^d}  \abs{\widehat{\varphi}_{j}(\bm{\omega}+ 2\pi (\mathbf{D}^T)^{-j} \bm{k})}^{2n}   }.
\end{equation*}
For any fundamental domain $\Omega_j$ of the lattice $2\pi(\mathbf{D}^T)^{-j}\mathbb{Z}^d$, let $\chi_{\Omega_j}$ denote the associated characteristic function. Proving decorrelation  depends on showing that the functions 
$m_{n,j}$ converge almost everywhere to some characteristic function $\chi_{\Omega_j}$. This analysis is closely related to the convergence of cardinal series as studied in \cite{deboor86}.  Our proof relies on the techniques used by Baxter to prove the convergence of the Lagrange functions associated with multiquadric functions \cite[Chapter 7]{baxter92}.  The idea is to define disjoint sets covering $\mathbb{R}^d$. Each set has a single point in any given fundamental domain, and we analyze the convergence of $m_{n,j}$ on these sets.
\begin{definition}\label{def:funddom}
Let $\Omega_j$ be a fundamental domain of $2\pi(\mathbf{D}^T)^{-j}\mathbb{Z}^d$. For each $j\in \mathbb Z$ and for each $\bm{x}\in \Omega_j$, define the set
\begin{equation*}
E_{j,\bm{x}}
:=
\left\{\bm{x}+2\pi (\mathbf{D}^T)^{-j} \bm{k}:\bm{k} \in \mathbb{Z}^d\right\}.
\end{equation*}
Since $\phi_{1,j}$ generates a Riesz basis, each set $E_{j,\bm{x}}$ has a finite number of elements $\bm{y}$ with $m_{1,j}(\bm{y})$ of maximal size. We define $F_j$ to be the set of $\bm{x}\in\Omega_j$ such that there is not a unique $\bm{y}\in E_{j,\bm{x}}$ where  $m_{1,j}$ attains a maximum; i.e., $\bm{x}$ is in the complement of $F_j$ if there exists $\bm{y}\in E_{j,\bm{x}}$ such that 
\begin{equation*}
m_{1,j}(\bm{y}) 
> 
m_{1,j}\left(\bm{y}+2\pi (\mathbf{D}^T)^{-j} \bm{k}\right) 
\end{equation*}
for all $\bm{k}\neq \bm{0}$.
\end{definition}
\begin{lemma}\label{lem:decor}
Let $\bm{x} \in \Omega_j\backslash F_j$, then for $ \bm{y}\in E_{j,\bm{x}} $  we have  $m_{n,j}(\bm{y})\rightarrow 0$  if and only if  $m_{1,j}(\bm{y})$ is not of maximal size over $E_{j,\bm{x}}$.  Furthermore, if  $m_{1,j}(\bm{y})$ is of maximal size, then $m_{n,j}(\bm{y})\rightarrow 1$.
\end{lemma}
\begin{proof} 
Fix $\bm{x}\in \Omega_j\backslash F_j$ and $y\in E_{j,\bm{x}}$.  Notice that the periodicity of the denominator of $m_{1,j}$ implies that $m_{1,j}(\bm{y})$ is maximal iff $\abs{\widehat{\varphi}_{1,j}(\bm{y})}$ is maximal.  

Let us first suppose $m_{1,j}(\bm{y})$ is not maximal.  If $\abs{\widehat{\varphi}_{1,j}(\bm{y})}=0$, the result is obvious.  Otherwise, there is some $\bm{k}_0 \in \mathbb{Z}^d$  and $b<1$ such that
\begin{equation*}
\abs{\widehat{\varphi}_{1,j}(\bm{y})}  
\leq  
b \abs{\widehat{\varphi}_{1,j}(\bm{y}+2\pi (\mathbf{D}^T)^{-j}\bm{k}_0)}. 
\end{equation*}
Therefore
\begin{align*}
\abs{\widehat{\varphi}_{n,j}(\bm{y})}^2  
&\leq  
b^{2n} \abs{\widehat{\varphi}_{n,j}(\bm{y}+2\pi (\mathbf{D}^T)^{-j}\bm{k}_0)}^2 \\
&\leq  
b^{2n} \sum_{\bm{k} \in \mathbb{Z}^d} \abs{\widehat{\varphi}_{n,j}(\bm{y}+2\pi(\mathbf{D}^T)^{-j} \bm{k})} ^2
\end{align*}
and the result follows.

Next, suppose  $m_{1,j}(\bm{y})$ is of maximal size. Since $\widehat{\varphi}_{1,j}$ has no periodic zeros, $\abs{\widehat{\varphi}_{1,j}(\bm{y})}\neq 0$.  Therefore
\begin{equation*}
m_{n,j}(\bm{y})
= 
\frac{1}{B_{n,j}(\bm{y})}
\end{equation*}
with 
\begin{equation*}
B_{n,j}(\bm{y})
= 
\sum_{\bm{k} \in \mathbb{Z}^d} \abs{ \frac{\widehat{\varphi}_{1,j}(\bm{y}+2\pi (\mathbf{D}^T)^{-j} \bm{k})}{\widehat{\varphi}_{1,j}(\bm{y})} }^{2n}.
\end{equation*}
Since $\abs{\widehat{\varphi}_{1,j}(\bm{y})}$ is of maximal size, all terms of the sum except one are less than $1$. In particular, $B_{n,j}$ will converge to $1$ as $n$ increases. 
\end{proof}
\begin{lemma}
 Let $j \in \mathbb Z$, and let $\Omega_j$ be a fundamental domain. If the Lebesgue measure of $F_j$ is $0$, then
\begin{equation*}
 \sum_{\bm{k} \in \mathbb{Z}^d} m_{n,j}(\cdot+2\pi (\mathbf{D}^T)^{-j}\bm{k})^2
\end{equation*}
converges to $\chi_{\Omega_j}$ in  $ L_1(\Omega_j)$ as $n\rightarrow \infty$.
\end{lemma}
\begin{proof}
The sum is bounded above by $1$, so Lemma \ref{lem:decor} implies that it converges to $\chi_{\Omega_j}$ on the complement of $F_j$. Hence, we apply the dominated convergence theorem to obtain the result.   
\end{proof}

With this theorem, we show how the wavelets corresponding to $\Lop_n$ decorrelate within scale as $n$ becomes large.  The way we characterize decorrelation is in terms of the semi-inner products
\begin{equation*}
(f,g)_{n}
:=
\int_{\mathbb{R}^d} \widehat{f} \widehat{g}^* \abs{\widehat{L}_{n}}^{-2},
\end{equation*}
which are true inner products for the wavelets
\begin{equation*}
\psi_{n,j+1}
=
\Lop_{n}^*\phi_{n,j}.
\end{equation*}
\begin{theorem}\label{th:dc}
Suppose the Lebesgue measure of $\cup_{j\in \mathbb{Z}} F_j$ is $0$, where $F_j$ is from  Definition \ref{def:funddom}. Then as $n$ increases, the wavelet coefficients decorrelate in the following sense. For any $j\in\mathbb Z, \bm{k}\in \mathbb{Z}^d\backslash \{0\}$ we have
\begin{equation*}
 (\psi_{n,j+1},\psi_{n,j+1}(\cdot- \mathbf{D}^{j}\bm{k} ))_n \rightarrow 0   
\end{equation*}
as $n \rightarrow \infty$.
\end{theorem}
\begin{proof}
First, we express the inner product as an integral
\begin{align*}
 (\psi_{n,j+1},\psi_{n,j+1}(\cdot-  \mathbf{D}^{k}\bm{k}))_n 
 &= 
 \int_{\mathbb R^d} \widehat{\psi}_{n,j+1}(\bm{\omega}) (\psi_{n,j+1}(\cdot-\mathbf{D}^{j}\bm{k}))^\wedge(\bm{\omega})^*  \abs{\widehat{L}_n}^{-2} \rmd\bm{\omega}\\
 &= 
 \abs{\det(\mathbf{D})}^{2j}\int_{\mathbb R^d} m_{n,j}(\bm{\omega})^2e^{- i\mathbf{D}^{j}\bm{k}\cdot\bm{\omega}}\rmd\bm{\omega},
\end{align*}
and we periodize the integrand to get
\begin{align*}
\int_{\mathbb R^d} m_{n,j}(\bm{\omega})^2e^{- i\mathbf{D}^{j}\bm{k}\cdot\bm{\omega}}d\bm{\omega} 
&=
\sum_{\bm{l}\in 2\pi (\mathbf{D}^T)^{-j}\mathbb{Z}^d} \int_{\Omega_j +\bm{l}} m_{n,j}(\bm{\omega})^2e^{- i\mathbf{D}^{j}\bm{k}\cdot\bm{\omega}} \rmd\bm{\omega} \\
&= 
\int_{\Omega_j}  e^{- i\mathbf{D}^{j}\bm{k}\cdot\bm{\omega}}\sum_{\bm{l}\in 2\pi (\mathbf{D}^T)^{-j}\mathbb{Z}^d}  m_{n,j}(\bm{\omega}-\bm{l})^2\rmd\bm{\omega}. 
\end{align*}
The last expression converges to $0$ by the Lebesgue dominated convergence theorem. 
\end{proof}

Let us now show how this result implies decorrelation of the wavelet coefficients.  Recall that our model for a random signal $s$ is based on the stochastic differential equation $\Lop s =w$, where $w$ is a non-Gaussian white noise \cite{unser_part1} and the operator $\Lop$ is spline admissible. We denote the wavelet coefficients as
\begin{align*}
c_{n,j+1,\bm{k}} 
&:= 
\left<s,\psi_{n,j+1}(\cdot-\mathbf{D}^j\bm{k}) \right> \\
&=  
\left< w, \phi_{n,j}(\cdot-\mathbf{D}^j\bm{k}) \right>.
\end{align*}
Our stochastic model implies that the coefficients are random variables.  Hence, for distinct $\bm{k}$ and $\bm{k}'$, the covariance between $c_{n,j+1,\bm{k}}$ and $c_{n,j+1,\bm{k}'}$ is determined by the expected value of their product:
\begin{equation*}
\mathbb{E}\{ c_{n,j+1,\bm{k}} c_{n,j+1,\bm{k}'}     \} 
= 
\mathbb{E}\left\{ \left<w,\phi_{n,j}(\cdot-\mathbf{D}^j\bm{k}) \right>\left<w,\phi_{n,j}(\cdot-\mathbf{D}^j\bm{k}')\right> \right\}.
\end{equation*} 
As long as the white noise $w$ has zero mean and finite second-order moments, the covariance satisfies
\begin{align*}
\mathbb{E}\{ c_{n,j+1,\bm{k}} c_{n,j+1,\bm{k}'}     \} 
&= 
\left<\phi_{n,j}(\cdot-\mathbf{D}^j\bm{k}), \phi_{n,j}(\cdot-\mathbf{D}^j\bm{k}')  \right> \\
&=  
(\psi_{n,j+1}(\cdot- \mathbf{D}^{j}\bm{k} ),\psi_{n,j+1}(\cdot- \mathbf{D}^{j}\bm{k}' ))_n \\
&\rightarrow 
0,
\end{align*}
where the convergence follows from Theorem \ref{th:dc}.  Therefore, when the standard deviations of $c_{n,j+1,\bm{k}}$ and $c_{n,j+1,\bm{k}'}$ are bounded below, the correlation between the coefficients converges to zero.

\section{Discussion and Examples}\label{sec:conclusion}

The formulation presented in this paper is quite general and accommodates many operators.  In this section, we show how it relates to previous wavelet constructions, and we provide examples that are not covered by previous theories.

\subsection{Connection to previous constructions}\label{sec:connection}

Our operator-based wavelet construction can be viewed as a direct generalization of the cardinal spline wavelet construction.  To see this, define the B-spline $N_1$ to be the characteristic function of the interval $[0,1]\subset \mathbb{R}$. Then for $m=2,3,\dots$, let the B-splines $N_m$ be defined by
\begin{equation*}
N_m(x) 
:= 
\int_0^1 N_{m-1} (x-t) {\rm d} t.
\end{equation*}
The fundamental (cardinal) interpolatory spline $\phi_{2m}$ is then defined as the linear combination
\begin{equation}\label{eq:interp_bspline}
\phi_{2m} (x) 
:= 
\sum_{k\in\mathbb{Z}}  \alpha_{k,m} N_{2m} (x+m-k),
\end{equation}
satisfying the interpolation conditions:
\begin{equation*}
\phi_{2m}(k)
=
\delta_{k,0}, \quad k\in\mathbb{Z},
\end{equation*}
where $\delta$ denotes the Kronecker delta function.  In \cite{chui91}, the authors define the cardinal B-spline wavelets (relative to the scaling function $N_m$) as
\begin{equation*}
\psi_m(x) 
:=  
\left(\frac{{\rm d}^m}{{\rm d} x^m}\phi_{2m} \right) (2x-1)
\end{equation*}

Within the context of our construction, the operator $\Lop$ is a constant multiple of the order $m$ derivative, and $\Lop^*\Lop$ is a constant multiple of the order $2m$ derivative.  Therefore our $\Lop^*\Lop$ spline interpolants \eqref{eq:lagint} are equivalent to the $\phi_{2m}$ defined in \eqref{eq:interp_bspline}, so we obtain the same wavelet spaces. In particular,  $\psi_m$ is a constant multiple of $\Lop^* \phi_{2m}$.  

A more general construction is given in \cite{micchelli91}.  In that paper, the the authors allow for scaling functions $\varphi$ that are defined in the Fourier domain by $\widehat{\varphi} = T/q$, where $T$ is a trigonometric polynomial 
\begin{equation*}
T(\bm{\omega}) 
:= 
\sum_{\bm{k}\in \mathbb{Z}^d} c[\bm{k}] e^{-i \bm{k}\cdot  \bm{\omega}}, \quad \bm{\omega} \in \mathbb{R}^d
\end{equation*}
and $q$  is a homogeneous polynomial
\begin{equation*}
q(\bm{\omega}) 
:=
\sum_{\abs{\bm{k}}=m} q_{\bm{k}} \bm{\omega}^{\bm{k}},\quad \bm{\omega} \in \mathbb{R}^d
\end{equation*}
of degree $m$ with $m>d$. Here, $q$ is also required to be elliptic; i.e., $q$ can only be zero at the origin.  The authors then define the Lagrange function $\phi$ by
\begin{equation*}
\widehat{\phi} 
= 
\frac{\abs{\widehat{\varphi}}^2}{\sum_{\bm{k}\in\mathbb{Z}^d} \abs{\widehat{\varphi}(\cdot+2\pi \bm{k})  }^2 },
\end{equation*}
and they define an elliptic spline wavelet $\psi_0$ as
\begin{equation*}
\widehat{\psi}_0 
= 
2^{-d} q^* \widehat{\phi}(2^{-1}\cdot).
\end{equation*}
Thus, we can see that our construction is also a generalization of the elliptic spline wavelet construction,  the primary extension being that we allow for a broader class of operators.  In fact, both of the prior constructions use scaling functions associated with operators that have homogeneous symbols.  This special case has the  property that the multiresolution spaces can be generated by dilation.  
\begin{proposition}\label{pr:homogen}
Let $\Lop$ be an admissible Fourier multiplier operator whose symbol $\widehat{L}$ is positive (except at the origin), continuous, and homogeneous of order $\alpha>d/2$; i.e.,  
\begin{equation*}
\widehat{L}(a\bm{\omega}) 
= 
a^\alpha\widehat{L}(\bm{\omega}), \quad \text{for}\quad a>0.
\end{equation*}
Further assume that there is a localization operator $\Lop_{{\rm d},0}$ (of the form described in Definition \ref{def:spadm}) such that the generalized B-spline $\varphi_0$, defined by
\begin{equation*}
\widehat{\varphi}_0 (\bm{\omega})
:= 
\frac{\widehat{L}_{{\rm d},0}(\bm{\omega})}{\widehat{L}(\bm{\omega})},
\end{equation*}
satisfies the Riesz basis condition 
\begin{equation*}
0
<
A
\leq 
\sum_{\bm{k}\in\mathbb{Z}^d}\abs{\widehat\varphi_0(\bm{\omega}+2\pi \bm{k})}^2
\leq 
B 
< 
\infty,
\end{equation*}
for some $A$ and $B$ in $\mathbb{R}$. Then the pair $\Lop, \mathbf{D}=2\mathbf{I}$ is spline admissible of order $\alpha$.
\end{proposition}
\begin{proof}
The first two conditions of Definition \ref{def:spadm} are automatically satisfied.  For the third condition, let $C_{\Lop}>0$ be a constant satisfying
\begin{equation*}
C_{\Lop} \abs{\widehat{L}(\bm{\omega})}^2 
\geq 
1
\end{equation*} 
on the unit sphere $\mathbb{S}^{d-1}\subset\mathbb{R}^d$. Then for $\abs{\bm{\omega}}>0$, homogeneity of $\widehat{L}$ implies 
\begin{align*}
C_{\Lop} \left(1+ \abs{\widehat{L}(\bm{\omega})}^2 \right) 
&\geq 
C_{\Lop} \abs{\widehat{L} \left( \abs{\bm{\omega}} \frac{\bm{\omega}}{\abs{\bm{\omega}}}    \right)}^2 \\
& \geq 
\abs{\bm{\omega}}^{2\alpha}.
\end{align*} 
For the fourth property, we let $\widehat{L}_{{\rm d},0}(\bm{\omega})= \sum_{\bm{k}\in \mathbb{Z}^d}p[\bm{k}]e^{i\bm{\omega}\cdot \bm{k}}$ and define
\begin{equation*}
\widehat{L}_{{\rm d},j}(\bm{\omega}) 
:= 
2^{-j\alpha}\sum_{\bm{k}\in \mathbb{Z}^d}p[\bm{k}]e^{i\bm{\omega}\cdot 2^j\bm{k}}.
\end{equation*}
The generalized B-splines $\varphi_j$ will then satisfy
\begin{align*}
\widehat{\varphi}_j(\bm{\omega}) 
&= 
\widehat{L}_{{\rm d},j}(\bm{\omega}) \widehat{L}(\bm{\omega})^{-1}\\
&= 
2^{-j\alpha}\widehat{L}_{{\rm d},0}(2^j\bm{\omega}) \widehat{L}(\bm{\omega})^{-1}\\
&= 
\widehat{L}_{{\rm d},0}(2^j\bm{\omega}) \widehat{L}(2^{j}\bm{\omega})^{-1}\\
&= 
\widehat{\varphi}_0(2^j\bm{\omega}).
\end{align*}
The Riesz basis property can then be verified, since
\begin{equation*}
\sum_{\bm{k}\in\mathbb{Z}^d}\abs{\widehat\varphi_j(\bm{\omega}+2\pi 2^{-j}\bm{k})}^2 
= 
\sum_{\bm{k}\in\mathbb{Z}^d}\abs{\widehat\varphi_0(2^{j}\bm{\omega}+2\pi \bm{k})}^2.
\end{equation*}
\end{proof}

In summary, our wavelet construction generalizes these known constructions for homogeneous Fourier multiplier operators, and it accommodates the more complex setting of non-homogeneous operators.

\subsection{Mat\'ern and Laplace operator examples} \label{sec:examples_lap}

The $d$-dimensional Mat\'ern operator is not homogeneous, so it provides an example that is not included in traditional wavelet constructions. Its symbol is $\widehat{L}_\nu(\bm{\omega})=(\abs{\bm{\omega}}^2+1)^{\nu/2}$, with the parameter $\nu > d/2$. As $\widehat{L}_\nu(\bm{\omega})^{-1}$ satisfies the Riesz basis condition, no localization operator is needed. Therefore, the operator $\Lop_\nu$ is spline-admissible of order $\nu$ for any admissible subsampling matrix $\mathbf{D}$. 
 
Next, consider the iterated Laplacian operator with symbol $\widehat{L}=\abs{\bm{\omega}}^{2m}$, where $m>d/4$ is an integer, and let $\mathbf{D}=2\mathbf{I}$.  Localization operators $\Lop_{{\rm d},j}$ can be constructed as in \cite{chui92}, and all of the conditions of Definition \ref{def:spadm} are satisfied. 
 
Note that each of these operators satisfies the growth condition of Theorem \ref{th:rb2}, so the corresponding wavelet spaces may be used to construct Riesz bases of $L_2(\mathbb{R}^d)$.

\subsection{Construction of non-standard localization operators}
\label{sec:examples_helm}

Here, we consider the Helmholtz operator $\Lop$ and construct corresponding localization operators $\Lop_{\rmd,j}$.  While we focus on this particular operator, the presented method is sufficient to be applied more generally.

The Helmholtz operator is defined by its symbol $\widehat{L}(\omega)=1/4-\abs{\bm{\omega}}^2$.    
The wavelets corresponding to $\Lop$ could potentially be applied in optics, as the Helmholtz equation,
\begin{equation*}
\Delta u + \lambda u 
=
f,
\end{equation*}
is a reduced form of the wave equation \cite[Chapter 5]{duffy01}.
In what follows, we show that this is a spline-admissible operator for the scaling matrix $\mathbf{D}=2\mathbf{I}$ on $\mathbb{R}^2$. However, since the wavelets $\psi_{j+1}=\Lop^*\phi_j$ do not form a Riesz basis for the coarse-scale wavelet spaces, we only consider $j\leq 0$.

Our construction of localization operators $\Lop_{\rmd, j}$ is based on the fact that sufficiently smooth functions have absolutely convergent Fourier series \cite[Theorem 3.2.9]{grafakos08}. This implies that we can define $\Lop_{\rmd, j}$ by constructing smooth, periodic functions $\widehat{L}_{\rmd, j}$ that are asymptotically equivalent to $\widehat{L}$ at its zero set. In fact, we define $\widehat{L}_{\rmd, j}$ to be equal to $\widehat{L}$ near its zeros.

Notice that $\widehat{L}$ is zero on the circle of radius $1/4$ centered at the origin, and it is smooth in a neighborhood of this circle.  Therefore, in the case $j=0$, we choose $\epsilon>0$ sufficiently small and define $\widehat{L}_{\rmd,0}$ to be a function satisfying:
\begin{enumerate}
\item $\widehat{L}_{\rmd,0}(\bm{\omega})=\widehat{L}(\bm{\omega})$ for $\abs{\abs{\bm{\omega}}-1/4}<\epsilon$;
\item  $\widehat{L}_{\rmd,0}(\bm{\omega})$ is constant for $\abs{\abs{\bm{\omega}}-1/4}>3\epsilon$ and $\omega \in [-\pi,\pi]^2$;
\item $\widehat{L}_{\rmd,0}(\bm{\omega})$ is periodic with respect to the lattice $2\pi \mathbb{Z}^2$.
\end{enumerate}
Such a function can be constructed using a smooth partition of unity $\{f_n\}_{n=1}^N$ on the torus, where each $f_n$ is supported on a ball of radius $\epsilon$.  Here, we require each $f_n$ to be positive, and the partition of unity condition means that  
\begin{equation*}
\sum_{n=1}^N f_n(\bm{\omega})
=
1.
\end{equation*} 
We partition the index set $\{1,\dots,N\}$ into the three subsets $\Lambda_1,\Lambda_2,\Lambda_3$ as follows:
\begin{enumerate}
\item If the support of $f_n$ has a non-empty intersection with the annulus $\abs{\abs{\bm{\omega}}-1/4}\leq \epsilon$, then $n\in \Lambda_1$;
\item Else if the support of $f_n$ lies in the ball of radius $1/4-\epsilon$ centered at the origin, then $n\in \Lambda_2$;
\item Else $n\in \Lambda_3$.
\end{enumerate}
We now define the periodic function
\begin{equation*}
\widehat{L}_{\rmd,0} 
:= 
\sum_{n\in \Lambda_1} f_n \widehat{L} + \sum_{n\in \Lambda_2} f_n  -  \sum_{n\in \Lambda_3} f_n,
\end{equation*}
on $[-\pi,\pi]^2$, and it can be verified that this function has the required properties.

Using a similar approach, we can define $\widehat{L}_{\rmd, j}$ for $j<0$, and the conditions of spline admissibility can be verified. Since the Helmholtz operator satisfies the conditions of Theorem \ref{th:rb2}, the resulting wavelet system is a Riesz basis of $L_2(\mathbb{R}^d)$. 

In conclusion, we have constructed localization operators (and hence generalized B-splines) for the Helmholtz operator. Furthermore, the presented method applies in greater generality
to operators whose symbols are smooth near their zero sets.

\appendix
\renewcommand\thesection{Appendix \Alph{section}}

\section{Discrete Sums}

Let $X=\{\bm{x}_k\}_{k\in\mathbb{N}}$ be a countable collection of points in $\mathbb{R}^d$, and define
\begin{align*}
h_X 
&:= 
\sup_{\bm{x} \in \mathbb{R}^d} \inf_{k\in \mathbb{N}} \abs{\bm{x}-\bm{x}_k} \\
q_X 
&:= 
\frac{1}{2} \inf_{k\neq k'} \abs{\bm{x}_k-\bm{x}_{k'}}.
\end{align*}
Also, let $B(\bm{x},r)$ denote the ball of radius $r$ centered at $\bm{x}$.
Proving Riesz bounds for the wavelet spaces relies on the following propositions concerning sums of function values over discrete sets.
\begin{proposition}\label{pr:lattice_lower}
If $h_X<\infty$ and $r>d/2$, then there exists a constant $C>0$ (depending only on $r$ and $d$, not $h_X$) such that  
\begin{equation*}
\sum_{\abs{\bm{x}_k} \geq 2h_X} \abs{\bm{x}_k}^{-2r} 
\geq 
C h_X^{-2r}
\end{equation*}
\end{proposition}
\begin{proof}
For $\abs{\bm{x}_k}\geq 2h_X$, we have
\begin{equation*}
\abs{\bm{x}_k -h_X \frac{\bm{x}_k}{\abs{\bm{x}_k}} } 
\geq 
2^{-1}\abs{\bm{x}_k},
\end{equation*}
which implies
\begin{equation*}
\abs{\bm{x}_k}^{-2r} 
\geq 
2^{-2r} \abs{\bm{x}_k -h_X \frac{\bm{x}_k}{\abs{\bm{x}_k}} }^{-2r}.
\end{equation*}
Then
\begin{align*}
\sum_{\abs{\bm{x}_k}\geq 2h_X} \abs{\bm{x}_k}^{-2r}
&\geq 
2^{-2r} \sum_{\abs{\bm{x}_k}\geq 2h_X} \abs{\bm{x}_k -h_X \frac{\bm{x}_k}{\abs{\bm{x}_k}} }^{-2r} \frac{\text{Vol}\left(B(\bm{x}_k,h_X)  \right)}{\text{Vol}\left(B(\bm{x}_k,h_X)  \right)}  \\
&\geq 
\frac{2^{-2r}}{\text{Vol}\left(B(\bm{0},h_X) \right)} \sum_{\abs{\bm{x}_k} \geq 2h_X} \int_{B(\bm{x}_k,h_X)} \abs{\bm{x}}^{-2r} {\rm d}\bm{x}\\
&\geq 
Ch_X^{-d} \int_{3h_X}^{\infty} t^{-2r+(d-1)}{\rm d}t \\
&\geq 
Ch_X^{-2r}
\end{align*}
\end{proof}
\begin{proposition}\label{pr:lattice_upper}
If $r>d/2$ and $\abs{\bm{x}_k} \geq q_X/2$ for all $k\in\mathbb{N}$, then there exists a constant $C>0$ (depending only on $r$ and $d$, not $q_X$) such that  
\begin{equation*}
\sum_{k\in\mathbb{N}} \abs{\bm{x}_k}^{-2r} 
\leq 
C q_X^{-2r}
\end{equation*}
\end{proposition}
\begin{proof}
Using the fact that $\abs{\bm{x}_k} \geq q_X/2$, we can write
\begin{equation*}
\abs{\bm{x}_k +\frac{q_X}{4} \frac{\bm{x}_k}{\abs{\bm{x}_k}} } 
\leq 
2\abs{\bm{x}_k},
\end{equation*}
which implies 
\begin{equation*}
\abs{\bm{x}_k}^{-2r} 
\leq 
2^{2r} \abs{\bm{x}_k +\frac{q_X}{4} \frac{\bm{x}_k}{\abs{\bm{x}_k}}}^{-2r}.
\end{equation*}
We now have
\begin{align*}
\sum_{k\in\mathbb{N}} \abs{\bm{x}_k}^{-2r} 
&\leq 
2^{2r} \sum_{k\in\mathbb{N}} \abs{\bm{x}_k +\frac{q_X}{4} \frac{\bm{x}_k}{\abs{\bm{x}_k}}}^{-2r}\frac{\text{Vol}\left(B(\bm{x}_k,q_X/4)  \right)}{\text{Vol}\left(B(\bm{x}_k,q_X/4)  \right)} \\
&\leq 
\frac{2^{2r}}{\text{Vol}\left(B(\bm{0},q_X/4)  \right)} \sum_{k\in\mathbb{N}} \int_{B(\bm{x}_k,q_X/4)} \abs{\bm{x}}^{-2r} {\rm d}\bm{x} \\
&\leq 
\frac{2^{2r}}{\text{Vol}\left(B(\bm{0},q_X/4)  \right)} \int_{\abs{\bm{x}}>q_X/4} \abs{\bm{x}}^{-2r} {\rm d}\bm{x} \\
&\leq 
Cq_X^{-d} \int_{q_X/4}^{\infty} t^{-2r+(d-1)}{\rm d}t\\
&\leq 
C q_X^{-2r}.
\end{align*}
\end{proof}

\bibliographystyle{plain}
\bibliography{arxiv_op_wave_3}
\end{document}